\documentclass[11pt]{article}

\usepackage{amssymb,amsmath,euscript,bbm,xcolor,graphicx,epstopdf}
\RequirePackage[numbers]{natbib}

\usepackage{amsmath,amssymb,amsthm,bm,euscript,bbm}
\usepackage{verbatim}
\usepackage{graphicx}
\usepackage{epstopdf}
\usepackage{color}
\usepackage{pstricks,pst-node,pst-plot}
\usepackage{rotating}
\usepackage{enumitem}
\newtheorem{proposition}{Proposition}[section]
\newtheorem{lemma}[proposition]{Lemma}
\newtheorem{theorem}[proposition]{Theorem}

\def\la{\lambda}
\def\La{\Lambda}
\def\ep{\varepsilon}

\def\l{{\langle}}
\def\r{\rangle}

\newcommand{\ol}{\overline}
\def\R{{\mathbb R}}

\def\E{{\mathbb E}}
\def\P{{\mathbb P}}

\makeatletter \@addtoreset{equation}{section} \makeatother
\newenvironment{remark}{%
    \vspace{0.3cm} \pagebreak [2]%
    \par%
    \refstepcounter{proposition}
    \noindent%
    {\bf Remark~\theproposition\  }}{\qed}%
\begin{document}

\title {The expected Euler characteristic approximation to excursion probabilities of Gaussian vector fields}
\author{Dan Cheng \thanks{Partially supported by NSF Grant DMS-1902432 and Simons Foundation Collaboration Grant 854127.} \\ Arizona State University
 \and Yimin Xiao \thanks{Partially supported  by the NSF grants DMS-1855185 and DMS-2153846. } \\ Michigan State University }

\maketitle

\begin{abstract}
	Let $\{(X(t), Y(s)): t\in T, s\in S\}$ be an $\R^2$-valued, centered, unit-variance smooth Gaussian vector field, where $T$ and $S$ are compact
	rectangles in $\R^N$. It is shown that, as $u\to \infty$, the joint excursion probability $\P \{\sup_{t\in T} X(t) \geq u, \sup_{s\in S} Y(s) \geq u \}$ 
	can be approximated by $\E\{\chi(A_u)\}$, the expected Euler characteristic of the excursion set $A_u=\{(t,s)\in T\times S: X(t) \ge u, Y(s) \ge u\}$,
	 such that the error is super-exponentially small. This verifies the expected Euler characteristic heuristic (cf.  Taylor, Takemura and Alder (2005), Alder and 
	 Taylor (2007)) for a large class of smooth Gaussian 
	 vector fields. 
\end{abstract}

\noindent{\small{\bf Keywords}: Gaussian vector fields, excursion probability, excursion set, Euler characteristic, correlation, 
asymptotics, EEC, super-exponentially small.}

\noindent{\small{\bf Mathematics Subject Classification}:\ 60G15, 60G60, 60G70.}

\section{Introduction}
 For a real-valued Gaussian random field $\{Z(t), t\in \R^N\}$ and a compact rectangle $T \subset \R^N$,  
 the excursion probability $\P\{\sup_{t\in T}Z(t) \ge u\}$ is a classical and very important problem in both 
 probability and statistics due to the vast applications in many areas such as $p$-value computations, 
 risk control and extreme event analysis, etc. Various methods for precise approximations of  
 $\P \{\sup_{t \in T} Z(t) \geq u \}$ have been developed. These include the double sum method, the tube
method, the Euler characteristic method, and the Rice method.
We refer to the monographs \citet{Piterbarg:1996}, \citet{Adler:2007}, 
\citet{AzaisW09} and the references therein for comprehensive accounts. However, extreme value theory of multivariate random fields (or random vector fields) is still under-developed and only a few authors have studied the joint excursion probability of multivariate random fields. Piterbarg and Stamatovic \cite{PS05} and Debicki et al.  \cite{DKMR10} established large deviation results for the excursion probability in multivariate case. Anshin \cite{Anshin06} obtained precise asymptotics for a special class of nonstationary bivariate Gaussian processes, under quite restrictive conditions. Hashorva and Ji \cite{HashorvaJi2014} and Debicki et al. 
\cite{Debicki:2015} derived precise asymptotics for the excursion probability of certain
multivariate Gaussian processes defined on the real line $\mathbb{R}$ with specific cross dependence structures.
Zhou and Xiao \cite{ZhouXiao07} studied the excursion probability of a class of non-smooth  bivariate Gaussian random fields 
by applying the double sum method. Their main results show explicitly that the excursion probabilities of bivariate Gaussian random fields depend not only on the
smoothness parameters of the coordinate fields  but also on their maximum cross-correlation.

In statistical applications, such joint excursion probability is the critical tool for constructing the simultaneous confidence region in a continuous-domain
approach \cite{Sun:2001}. In particular, motivated by the expected Euler characteristic (EEC) approximation to excursion probabilities of real-valued 
Gaussian random fields \cite{TTA05,Adler:2007}, we study in this work that the EEC approximation holds in general to the joint excursion probability of Gaussian vector fields. 

Let $\{(X(t), Y(s)): t\in T, s\in S\}$ be an $\R^2$-valued, centered, unit-variance smooth Gaussian vector field, 
where $T$ and $S$ are compact rectangles in $\R^N$. Let $A_u=\{(t,s)\in T\times S: X(t) \ge u, Y(s) \ge u\}$ be 
the excursion set where both components $X$ and $Y$ exceeding the level $u$. Our main objective is to show that, 
as $u\to \infty$, the joint excursion probability $\P \{\sup_{t\in T} X(t) \geq u, \sup_{s\in S} Y(s) \geq u \}$ can be 
approximated by $\E\{\chi(A_u)\}$, the EEC of $A_u$, such that the error is super-exponentially small; see Theorem 
\ref{Thm:MEC approximation je} below for precise description. This approximation result shows that the maximum correlation 
between  $X(t)$ and $Y(s)$, denoted by $R$ (see \eqref{eq:R} below), plays an important role in both $\E\{\chi(A_u)\}$ and the super-exponentially small error. Moreover, as we will see in the proof of Theorem \ref{Thm:MEC approximation je} (cf. $\mathcal{M}_0$ in \eqref{Eq:M0},  
$\mathcal{M}_1$ in \eqref{Eq:O-Udelta}, and $\mathcal{M}_2$ in \eqref{Eq:M2}),
the points where $R$ is attained make the major contribution for $\E\{\chi(A_u)\}$. Based on this 
observation, we also establish two simpler approximations in Theorem \ref{Thm:MEC approximation je2} under 
the boundary condition \eqref{Eq:boundary} on nonzero derivatives of the correlation function over boundary points where $R$ is attained and in Theorem 
\ref{Thm:MEC approximation} under the condition that there is only a unique point attaining $R$, respectively.

In general, the EEC approximation $\E\{\chi(A_u)\}$ can be expressed by the Kac-Rice formula as an integral; see \eqref{eq:EEC} in Theorem \ref{Thm:MEC approximation je}. In \cite{TTA05,Adler:2007}, the authors derived a nice expression for $\E\{\chi(A_u)\}$ called Gaussian kinematic formula, since they assumed that the real-valued Gaussian field has unit variance, which is an important condition to simplify the integration formula of $\E\{\chi(A_u)\}$. However, in our case here, the integration formula of $\E\{\chi(A_u)\}$ (see \eqref{eq:EEC}) mainly depends on the conditional correlation of $X(t)$ and $Y(s)$, 
which varies over $T\times S$. It turns to be very difficult to get an explicit expression for $\E\{\chi(A_u)\}$. Instead, one can apply the Laplace method to extract the term with the largest order of $u$ from the integral such that the remaining error is $o(1/u)\E\{\chi(A_u)\}$. To explain this, we show several examples on specific calculations in Section \ref{sec:example}. For an intuitive understanding on the EEC approximation, we may roughly treat the main term $\E\{\chi(A_u)\}$ as $g(u)e^{-u^2/(1+R)}$ (by approximating the integral in \eqref{eq:EEC}); and the error term $o(e^{-u^2/(1+R) - \alpha u^2})$ is super-exponentially small, where $g(u)$ is a polynomial in $u$ and $\alpha>0$ is some constant.

The paper is organized as follows. We introduce first the notations and assumptions in Section \ref{sec:notation}, and then state our main results Theorems \ref{Thm:MEC approximation je}, \ref{Thm:MEC approximation je2} and \ref{Thm:MEC approximation} in Section \ref{sec:main}. The proofs are then provided in three steps: (i) sketch the main ideas in Section \ref{sec:sketch}; (ii) study the super-exponentially small errors between the joint excursion probability and EEC in Sections \ref{sec:small} and \ref{sec:diff}; and (iii) provide final proofs for the main results in Section \ref{sec:proof}. Finally, we show in Section \ref{sec:example} several examples on evaluating the EEC and hence approximating the joint excursion probability explicitly.

\section{Notations and assumptions}\label{sec:notation}

Let $\{(X(t), Y(s)): t\in T, s\in S\}$ be an $\R^2$-valued, centered, unit-variance Gaussian vector field, where $T$ and $S$ are compact 
rectangles in $\R^N$. Let
\begin{equation}\label{eq:R}
\begin{split}
r(t,s)=\E\{X(t)Y(s)\}, \quad R=\sup_{t\in T,\,  s\in S}r(t,s).
\end{split}
\end{equation}
For a function $f(\cdot) \in C^2(\R^N)$ and $t\in \R^N$, let 
\begin{equation}\label{Eq:notatoin-diff}
\begin{split}
f_i (t)&=\frac{\partial f(t)}{\partial t_i}, \quad f_{ij}(t)=\frac{\partial^2 f(t)}{\partial t_i\partial t_j}, \quad \forall i, j=1, \ldots, N;\\
\nabla f(t) &= (f_1(t), \ldots , f_N(t))^{T}, \quad \nabla^2 f(t) = \left(f_{ij}(t)\right)_{ i, j = 1, \ldots, N}.
\end{split}
\end{equation}
For a symmetric matrix $B$, denote by $B \prec 0$ and $B \preceq 0$ if all the eigenvalues of $B$ are negative (i.e., $B$ is 
negative definite) and nonpositive (i.e., $B$ is  negative semi-definite), respectively. For two functions $f(x)$ and $g(x)$, we 
write $f(x)\sim g(x)$ as $x\to \infty$ if $\lim_{x\to \infty}f(x)/g(x)=1$.

Denote by $T=\prod^N_{i=1}[a_i, b_i]$ and $S=\prod^N_{i=1}[a'_i, b'_i]$, where $-\infty< a_i<b_i<\infty$ and $-\infty< a'_i<b'_i<\infty$. 
Following the notations in \citet[p.134]{Adler:2007},  we show below that $T$ and $S$ can be decomposed into unions of their interiors 
and the lower dimension faces. Based on these decompositions,  the Euler characteristic of the excursion set $A_u$ can be 
represented (see Section \ref{sec:main}).

A face $K$ of dimension $k$ is defined by fixing a subset $\sigma(K) \subset \{1, \ldots, N\}$ of size $k$ and a subset $\varepsilon(K) 
= \{\varepsilon_j, j\notin \sigma(K)\} \subset \{0, 1\}^{N-k}$ of size $N-k$ so that
\begin{equation*}
\begin{split}
K= \{ t=(t_1, \ldots, t_N) \in T: \, &a_j< t_j <b_j  \ {\rm if} \ j\in \sigma(K), \\
&t_j = (1-\ep_j)a_j + \ep_{j}b_{j}  \ {\rm if} \ j\notin \sigma(K)  \}.
\end{split}
\end{equation*}
Denote by $\partial_k T$ the collection of all $k$-dimensional faces in $T$. Then the interior of $T$ is denoted by $\overset{\circ}{T}=\partial_N T$ 
and the boundary of $T$ is given by $\partial T = \cup^{N-1}_{k=0}\cup _{K\in \partial_k T} K$. For each $t\in K\in \partial_k T$ and $s\in L\in \partial_l S$, let
\begin{equation*}
\begin{split}
\nabla X_{|K}(t) &= (X_{i_1} (t), \ldots, X_{i_k} (t))^T_{i_1, \ldots, i_k \in \sigma(K)}, \quad \nabla^2 X_{|K}(t) = (X_{mn}(t))_{m, n \in \sigma(K)}, \\
\nabla Y_{|L}(s) &= (Y_{i_1} (s), \ldots, Y_{i_l} (s))^T_{i_1, \ldots, i_l \in \sigma(L)}, \quad \nabla^2 Y_{|L}(s) = (Y_{mn}(s))_{m, n \in \sigma(L)}.
\end{split}
\end{equation*}
We can decompose $T$ and $S$  into  
\begin{equation*}
T= \bigcup_{k=0}^N \partial_k T = \bigcup_{k=0}^N\bigcup_{K\in \partial_k T}K,\qquad S= \bigcup_{l=0}^N \partial_l S = \bigcup_{l=0}^N\bigcup_{L\in \partial_l S}L,
\end{equation*}
respectively. For each $K\in \partial_k T$ and $L\in \partial_l S$, we define the \emph{number of extended outward maxima
above $u$} as
\begin{equation*}
\begin{split}
M_u^E (X,K) & := \# \{ t\in K: X(t)\geq u, \nabla X_{|K}(t)=0, \nabla^2 X_{|K}(t)\prec 0, \varepsilon^*_jX_j(t) \geq 0, \forall j\notin \sigma(K) \},\\
M_u^E (Y,L) & := \# \{ s\in L: Y(s)\geq u, \nabla Y_{|L}(s)=0, \nabla^2 Y_{|L}(s)\prec 0, \varepsilon^*_jY_j(s) \geq 0, \forall j\notin \sigma(L) \},
\end{split}
\end{equation*}
where $\ep^*_j=2\ep_j-1$, and define the \emph{number of local maxima above $u$} as
\begin{equation*}
\begin{split}
M_u (X,K) & := \# \{ t\in K: X(t)\geq u, \nabla X_{|K}(t)=0, \nabla^2 X_{|K}(t)\prec 0\},\\
M_u (Y,L) & := \# \{ s\in L: Y(s)\geq u, \nabla Y_{|L}(s)=0, \nabla^2 Y_{|L}(s)\prec 0 \}.
\end{split}
\end{equation*}
 Clearly, $M_u^E (X,K) \le M_u (X,K)$ and $M_u^E (Y,L) \le M_u (Y,L)$.

We shall make use of the following smoothness condition ({\bf H}1) and regularity conditions ({\bf H}2) and ({\bf H}3).
\begin{itemize}
	\item[({\bf H}1)] $X, Y \in C^2(\R^N)$ almost surely and their second derivatives satisfy the
	\emph{uniform mean-square H\"older condition}: there exist constants $C, \delta>0$ such that
	\begin{equation*}
	\begin{split}
	\E(X_{ij}(t)-X_{ij}(t'))^2 &\leq C \|t-t'\|^{2\delta}, \quad \forall t,t'\in T,\ i, j= 1, \ldots, N,\\
	\E(Y_{ij}(s)-Y_{ij}(s'))^2 &\leq C \|s-s'\|^{2\delta}, \quad \forall s,s'\in S,\ i, j= 1, \ldots, N.
	\end{split}
	\end{equation*}
	
	\item[({\bf H}2)] For every $(t, t', s)\in T^2\times S$ with $t\neq t'$, the Gaussian vector
	\begin{equation*}
	\begin{split}
	\big(&X(t), \nabla X(t), X_{ij}(t), X(t'), \nabla X(t'), X_{ij}(t'), \\
	&Y(s), \nabla Y(s), Y_{ij}(s), 1\leq i\leq j\leq N\big)
	\end{split}
	\end{equation*}
	is  non-degenerate; and for every $(s, s', t)\in S^2\times T$ with $s\neq s'$, the Gaussian vector
	\begin{equation*}
	\begin{split}
	\big(&Y(s), \nabla Y(s), Y_{ij}(s), Y(s'), \nabla Y(s'), Y_{ij}(s'), \\
	&X(t), \nabla X(t), X_{ij}(t), 1\leq i\leq j\leq N\big)
	\end{split}
	\end{equation*}
	is  non-degenerate.
	
	\item[({\bf H}3)] For every $(t, s)\in \partial_k T\times S$, $0\le k\le N-2$, such that $r(t,s)=R$, and 
	that the index set $\mathcal{I}^R_X(t,s) = \{ \ell: \frac{\partial r}{\partial t_\ell}(t,s)=0\}$ contains at least 
	two indices, the Hessian matrix
	\begin{equation}\label{eq:Hessian_I}
	\left(\frac{\partial^2 r}{\partial t_i\partial t_j}(t,s)\right)_{i,j\in \mathcal{I}^R_X(t,s) }\preceq 0.
\end{equation}
	For every $(t, s)\in  T\times \partial_l S$, $0\le l\le N-2$, such that $r(t,s)=R$, and that the index set 
	$\mathcal{I}^R_Y(t,s) = \{ \ell: \frac{\partial r}{\partial s_\ell}(t,s)=0\}$ contains at least two indices, the Hessian matrix
	\[
 \left(\frac{\partial^2 r}{\partial s_m\partial s_n}(t,s)\right)_{m,n\in \mathcal{I}^R_Y(t,s) } \preceq 0.
	\]
\end{itemize}

Although ({\bf H}3) looks technical, it is in fact a mild condition imposed only on the lower-dimension boundary points $(t,s)$ 
with $r(t,s)=R$. Roughly speaking, it shows that the correlation function should have a negative semi-definite Hessian matrix 
on boundary critical points where  the maximum correlation $R$ is attained. Since $r(t,s)=R$ implies $\frac{\partial r}{\partial t_\ell}(t,s) =0$ 
for all $\ell \in \sigma(K)$, we have $\mathcal{I}^R_X(t,s) \supset \sigma (K)$. Similarly, $\mathcal{I}^R_Y(t,s) \supset \sigma (L)$. 
We show below that, for $k=N-1$ or $k=N$, the property \eqref{eq:Hessian_I} is always satisfied.

(i) If $k = N$, then $t$ becomes a maximum point of $r$ (as a function of $t$) in the interior of $T$ and 
$\mathcal{I}^R_X(t,s) = \sigma (K) =\{1, \cdots, N\}$, implying \eqref{eq:Hessian_I}.

(ii) For $k=N-1$, we distinguish two cases. If $\mathcal{I}^R_X(t,s) = \sigma (K)$, then $t$ becomes a maximum point of 
$r$ restricted on $K$, hence \eqref{eq:Hessian_I} holds. If $\mathcal{I}^R_X(t,s) =\{1, \cdots, N\}$, let $s$ be fixed, it follows from Taylor's formula that
\begin{equation*}
r (t',s) =r (t,s) + (t'-t)^T \nabla^2 r(t,s) (t'-t)+o(\|t'-t\|^2), \quad t'\in T,
\end{equation*}
where $\nabla^2 r(t,s)$ is the Hessian with respect to $t$. Notice that $\{(t'-t)/\|t'-t\|: t'\in T\}$ contains all directions in $\R^N$ 
since $t\in K \in \partial_{N-1}T$, together with the fact $r(t,s)=R$, we see that $\nabla^2 r(t,s)$ cannot have any positive eigenvalue and hence \eqref{eq:Hessian_I} holds.

It is also evident from the 1D Taylor's formula that \eqref{eq:Hessian_I} holds if $\mathcal{I}^R_X(t,s)$ contains only one index. 
Combining these facts, together with the observations,
\begin{equation*}
	\begin{split}
		\frac{\partial r}{\partial t_i}(t,s) &= \E\{X_{i}(t)Y(s)\}, \quad \frac{\partial^2 r}{\partial t_i\partial t_j}(t,s) = \E\{X_{ij}(t)Y(s)\},\\
		\frac{\partial r}{\partial s_i}(t,s) &= \E\{X(t)Y_{i}(s)\}, \quad \frac{\partial^2 r}{\partial s_i\partial s_j}(t,s) = \E\{X(t)Y_{ij}(s)\},
	\end{split}
\end{equation*}
we obtain the following result.
\begin{proposition}\label{prop:H3}
	Under the condition $({\bf H}3)$, we have that, for every $(t, s)\in T\times S$ such that $r(t,s)=R$, the matrices
	\[
	(\E\{X_{ij}(t)Y(s)\})_{i,j\in \mathcal{I}^R_X(t,s) }\preceq 0 \quad {\rm and }  \quad (\E\{X(t)Y_{kl}(s)\})_{k,l\in \mathcal{I}^R_Y(t,s) } \preceq 0,
	\]
	where the index sets $\mathcal{I}^R_X(t,s)$ and $\mathcal{I}^R_Y(t,s)$ are defined respectively as
	\begin{equation*}
		\begin{split}
			\mathcal{I}^R_X(t,s) = \{ \ell: \E\{X_\ell(t)Y(s)\}=0\} \quad \text{\rm and } \quad \mathcal{I}^R_Y(t,s) = \{ \ell: \E\{X(t)Y_\ell(s)\}=0\}.
		\end{split}
	\end{equation*}
\end{proposition}


\section{Main results}\label{sec:main}
Here, we shall state our main results Theorems \ref{Thm:MEC approximation je}, \ref{Thm:MEC approximation je2} 
and \ref{Thm:MEC approximation}, whose proofs will be given in Section \ref{sec:proof}. Define respectively 
the excursion sets of $X$, $Y$ and $(X, Y)$ above level $u$ by 
\begin{equation*}
\begin{split}
A_u(X,T)&=\{t\in T: X(t) \geq u\},\\
A_u(Y,S)&=\{s\in S: Y(s) \geq u\} \quad {\text \rm and}\\
A_u:=A_u(X,T)\times A_u(Y,S)&=\{(t,s)\in T\times S: X(t) \geq u, Y(s) \geq u\}.
\end{split}
\end{equation*}
Let the \emph{number of extended outward critical points of index $i$
	above level $u$} be
\begin{equation*}
\begin{split}
\mu_i(X,K) & := \# \{ t\in K: X(t)\geq u, \nabla X_{|K}(t)=0, \text{index} (\nabla^2 X_{|K}(t))=i, \\
& \qquad \qquad \qquad \varepsilon^*_jX_j(t) \geq 0 \ {\rm for \ all}\ j\notin \sigma(K) \},\\
\mu_i(Y,L) & := \# \{ s\in L: Y(s)\geq u, \nabla Y_{|L}(s)=0, \text{index} (\nabla^2 Y_{|L}(s))=i, \\
& \qquad \qquad \qquad \varepsilon^*_jY_j(s) \geq 0 \ {\rm for \ all}\ j\notin \sigma(L) \}.
\end{split}
\end{equation*}
Recall that  $\ep^*_j=2\ep_j-1$ and the index of a matrix is defined as the number of its negative
eigenvalues. It follows from ({\bf H}1), ({\bf H}2) and the Morse theorem (see  Corollary 9.3.5 or 
pages 211--212 in \citet{Adler:2007}) that the Euler characteristic of the excursion set can be represented as
\begin{equation}\label{eq:Euler}
\begin{split}
\chi(A_u(X,T))&= \sum^N_{k=0}\sum_{K\in \partial_k T}(-1)^k\sum^k_{i=0} (-1)^i \mu_i(X,K),\\
\chi(A_u(Y,S))&= \sum^N_{l=0}\sum_{L\in \partial_l S}(-1)^l\sum^l_{i=0} (-1)^i \mu_i(Y,L).
\end{split}
\end{equation}
Since for two sets $D_1$ and $D_2$, $\chi(D_1\times D_2) = \chi(D_1)\chi(D_2)$, we have
\begin{equation}\label{Eq:Morse product}
\begin{split}
\chi(A_u)&=\chi(A_u(X,T)\times A_u(Y,S))= \chi(A_u(X,T))\times \chi(A_u(Y,S))\\
&=\sum^N_{k,l=0}\sum_{K\in \partial_k T, L\in \partial_l S}(-1)^{k+l}\bigg(\sum^k_{i=0} (-1)^i \mu_i(X, K)\bigg)\bigg(\sum^l_{j=0} (-1)^j \mu_j(Y, L)\bigg).
\end{split}
\end{equation}

Now we  state the following general result on the EEC approximation of the excursion probability.
\begin{theorem}\label{Thm:MEC approximation je} 
Let $\{(X(t),Y(s)): t\in T, s\in S\}$ be an $\R^2$-valued, centered, 
unit-variance Gaussian vector field satisfying $({\bf H}1)$, $({\bf H}2)$ and $({\bf H}3)$. Then there exists a constant
$\alpha>0$ such that as $u\to \infty$,
\begin{equation}\label{eq:EEC}
	\begin{split}
	&\quad \P \left\{\sup_{t\in T} X(t) \geq u, \sup_{s\in S} Y(s) \geq u \right\}\\
	&=\sum^N_{k,l=0}\sum_{K\in \partial_k T, L\in \partial_l S}(-1)^{k+l}\int_K\int_Ldtds\, p_{\nabla X_{|K}(t), \nabla Y_{|L}(s)}(0,0)  \E\big\{{\rm det}\nabla^2 X_{|K}(t){\rm det}\nabla^2 Y_{|L}(s)\\
	&\quad \times \mathbbm{1}_{\{X(t)\geq u, \ \varepsilon^*_\ell X_\ell(t) \geq 0 \ {\rm for \ all}\ \ell\notin \sigma(K)\}} \mathbbm{1}_{\{Y(s)\geq u, \ \varepsilon^*_\ell Y_\ell(s) \geq 0 \ {\rm for \ all}\ \ell\notin \sigma(L)\}} \big| \nabla X_{|K}(t)=\nabla Y_{|L}(s)=0\big\}\\
	&\quad  +o\left( \exp \left\{ -\frac{u^2}{1+R} -\alpha u^2 \right\}\right)\\
	&= \E\{\chi(A_u)\}+o\left( \exp \left\{ -\frac{u^2}{1+R} -\alpha u^2 \right\}\right).
	\end{split}
	\end{equation}
\end{theorem}

In general, the EEC approximation $\E\{\chi(A_u)\}$ is hard to compute, since the conditional expectation in 
\eqref{eq:EEC} involves the joint tail estimate and hence the conditional correlation on $X(t)$ and $Y(s)$, 
which varies over $T\times S$. However, one can apply the Laplace method to extract the term with the 
largest order of $u$ from $\E\{\chi(A_u)\}$ such that the remaining error is $o(1/u)\E\{\chi(A_u)\}$; see Section \ref{sec:example} for examples on this.
	
Note that, in \eqref{eq:EEC}, if $k=0$, then all terms involving $\nabla X_{|K}(t)$ and $\nabla^2 X_{|K}(t)$ in \eqref{eq:EEC} 
vanish. In particular, if $k=l=0$, then the integral in \eqref{eq:EEC} becomes a joint probability. We follow such notation in the results 
in Theorems \ref{Thm:MEC approximation je2} and \ref{Thm:MEC approximation} below as well. 

It can be seen from the proof of Theorem \ref{Thm:MEC approximation je} that those points attaining the maximal 
correlation $R$ make the major contribution for $\E\{\chi(A_u)\}$. Therefore, in many cases, the general EEC 
approximation $\E\{\chi(A_u)\}$ can be simplified.  The result below is based on the boundary condition 
\eqref{Eq:boundary} (which implies $({\bf H}3)$) on nonzero derivatives of the correlation function over boundary points where $R$ is attained.
\begin{theorem}\label{Thm:MEC approximation je2} Let $\{(X(t),Y(s)): t\in T, s\in S\}$ be an $\R^2$-valued, 
centered, unit-variance Gaussian vector field satisfying $({\bf H}1)$, $({\bf H}2)$ and the boundary condition 
\eqref{Eq:boundary}. Then there exists a constant $\alpha>0$ such that as $u\to \infty$,
	\begin{equation*}
	\begin{split}
	&\quad \P \left\{\sup_{t\in T} X(t) \geq u, \sup_{s\in S} Y(s) \geq u \right\}\\
	&=\sum^N_{k,l=0}\sum_{K\in \partial_k T, L\in \partial_l S}(-1)^{k+l}\int_K\int_L dtds\, p_{\nabla X_{|K}(t), \nabla Y_{|L}(s)}(0,0)  
	\E\big\{{\rm det}\nabla^2 X_{|K}(t){\rm det}\nabla^2 Y_{|L}(s)\\
	&\quad \times \mathbbm{1}_{\{X(t)\geq u, \ Y(s)\geq u\}}\big|\nabla X_{|K}(t)=\nabla Y_{|L}(s)=0\big\}  +o\left( \exp \left\{ -\frac{u^2}{1+R} -\alpha u^2 \right\}\right).
	\end{split}
	\end{equation*}
\end{theorem}

The following result is the asymptotic approximation for the special case when the correlation attains its maximum $R$ only at a unique point.

\begin{theorem}\label{Thm:MEC approximation} Let $\{(X(t),Y(s)): t\in T, s\in S\}$ be an $\R^2$-valued, centered, unit-variance Gaussian 
vector field satisfying $({\bf H}1)$, $({\bf H}2)$ and $({\bf H}3)$. Suppose that the correlation attains its maximum $R$ only at a single point 
$(t^*, s^*)\in K\times L$, where $K\in \partial_k T$ and $L\in \partial_l S$ with $k,l\ge 0$. Then there exists a constant $\alpha>0$ such that 
as $u\to \infty$,
	\begin{equation*}
		\begin{split}
			&\quad \P \left\{\sup_{t\in T} X(t) \geq u, \sup_{s\in S} Y(s) \geq u \right\}\\
			&=\sum_{J}\sum_{F}(-1)^{{\rm dim}(J)+{\rm dim}(F)}\int_J\int_F dtds\, p_{\nabla X_{|J}(t), \nabla Y_{|F}(s)}(0,0) \\
			&\quad \times\E\big\{{\rm det}\nabla^2 X_{|J}(t){\rm det}\nabla^2 Y_{|F}(s)\mathbbm{1}_{\{X(t)\geq u, \ \varepsilon^*_\ell X_\ell(t) \geq 0 \ {\rm for \ all}\ \ell\in \mathcal{I}^R_X(t^*,s^*)\setminus \sigma(J)\}}\\
			&\quad \times \mathbbm{1}_{\{Y(s)\geq u, \ \varepsilon^*_\ell Y_\ell(s) \geq 0 \ {\rm for \ all}\ \ell\in \mathcal{I}^R_Y(t^*,s^*)\setminus \sigma(F)\}} \big| \nabla X_{|J}(t)=\nabla Y_{|F}(s)=0\big\}\\
			&\quad +o\left( \exp \left\{ -\frac{u^2}{1+R} -\alpha u^2 \right\}\right)
		\end{split}
	\end{equation*}
where the sums are taken over all faces $J$ of $T$ such that $t^*\in \bar{J}$ and $\sigma(J)\subset \mathcal{I}^R_X(t^*,s^*)$, 
and all faces $F$ of $S$ such that $s^*\in \bar{F}$ and $\sigma(F)\subset \mathcal{I}^R_Y(t^*,s^*)$.
\end{theorem}

\section{Plan of the proofs}\label{sec:sketch}

Note that, for a smooth real-valued function $f$, $\sup_{t\in T} f(t) \geq u$ if and only if there exists at least one extended 
outward local maximum above $u$ on some face of $T$. Thus, under conditions $({\bf H}1)$ and $({\bf H}2)$, the following 
relation holds for each $u\in \R$:
\begin{equation}\label{Eq:maxima-faces}
\begin{split}
&\quad \left\{\sup_{t\in T} X(t) \geq u, \sup_{s\in S} Y(s) \geq u \right\} \\
&= \bigcup_{k, l=0}^N \bigcup_{K\in \partial_k T, \, L\in \partial_l S}
\{M_u^E (X,K) \geq 1, M_u^E (Y,L) \geq 1\} \quad {\rm a.s.}
\end{split}
\end{equation}
Therefore, we obtain the following upper bound for the joint excursion probability:
\begin{equation}\label{Ineq:upperbound je}
\begin{split}
&\quad \P\left\{\sup_{t\in T} X(t) \geq u, \sup_{s\in S} Y(s) \geq u \right\}\\
&\leq \sum_{k,l=0}^N\sum_{K\in \partial_k T, \, L\in \partial_l S} \P\{M_u^E (X,K) \geq 1, M_u^E (Y,L) \geq 1\}\\
&\leq \sum_{k,l=0}^N\sum_{K\in \partial_k T, \, L\in \partial_l S} \E \{M_u^E (X,K) M_u^E (Y,L) \}.
\end{split}
\end{equation}
On the other hand, notice that
\begin{equation*}
\begin{split}
&\quad \E\{M_u^E (X,K)M_u^E (Y,L) \} - \P\{M_u^E (X,K) \geq 1, M_u^E (Y,L) \geq 1\} \\
&= \sum_{i,j=1}^\infty (ij-1)\P\{M_u^E (X,K)=i, M_u^E (Y,L)=j\}\\
& \leq \sum_{i,j=1}^\infty [i(i-1)j+j(j-1)i]\P\{M_u^E (X,K)=i, M_u^E (Y,L)=j\}\\
&= \E \{M_u^E (X,K)[M_u^E (X,K)-1]M_u^E (Y,L) \}  + \E \{ M_u^E (Y,L)[M_u^E (Y,L)-1]M_u^E (X,K) \}
\end{split}
\end{equation*}
and
\begin{equation*}
\begin{split}
&\quad\P\{M_u^E (X,K) \geq 1, M_u^E (Y,L) \geq 1, M_u^E (X,K') \geq 1, M_u^E (Y,L') \geq 1\}\\
&\le \P\{M_u^E (X,K) \geq 1, M_u^E (Y,L) \geq 1, M_u^E (Y,L') \geq 1\}\\
&\le \E\{M_u^E (X,K) M_u^E (Y,L) M_u^E (Y,L')\}.
\end{split}
\end{equation*}
Combining these two inequalities with \eqref{Eq:maxima-faces} and applying the Bonferroni inequality, we 
obtain the following lower bound for the joint excursion probability:
\begin{equation}\label{Ineq:lowerbound je}
\begin{split}
&\quad \P \left\{\sup_{t\in T} X(t) \geq u, \sup_{s\in S} Y(s) \geq u \right\} \\
&\geq \sum_{k,l=0}^N\sum_{K\in \partial_k T, L\in \partial_l S} \Big\{ \E \{M_u^E (X,K)M_u^E (Y,L) \}- \\
&\quad \E \{M_u^E (X,K)[M_u^E (X,K)-1] M_u^E (Y,L) \} - \E \{ M_u^E (Y,L)[M_u^E (Y,L)-1]M_u^E (X,K) \} \Big\}\\
& \quad - \sum_{k,k',l=0}^N\sum_{\substack {K\in \partial_k T, L\in \partial_l S\\ K'\in \partial_{k'} T, K\neq K'}} \E\{M_u^E (X,K) M_u^E (X,K') M_u^E (Y,L)\}\\
&\quad- C_N\sum_{k,l,l'=0}^N\sum_{\substack{K\in \partial_k T, L\in \partial_l S\\ L'\in \partial_{l'} S, L\neq L'}} \E\{M_u^E (X,K) M_u^E (Y,L) M_u^E (Y,L')\},
\end{split}
\end{equation}
where $C_N$ is a constant depending only on $N$. 
\begin{remark}\label{remark:M_u}
	Note that, following the same arguments above, we have that the expectations on the number of extended outward maxima $M_u^E(\cdot)$ in both \eqref{Ineq:upperbound je} and \eqref{Ineq:lowerbound je} can be replaced by the expectations on the number of local maxima $M_u(\cdot)$.
\end{remark}

We call a function $h(u)$ \emph{super-exponentially small} [when compared with the joint excursion probability $\P \{\sup_{t\in T} X(t) \geq u, \sup_{s\in S} Y(s) \geq u \}$], if there exists a constant $\alpha >0$ such that $h(u) = o(e^{-\alpha u^2 - u^2/{(1+R)}})$ as $u \to \infty$. The main idea for proving the EEC approximation Theorem \ref{Thm:MEC approximation je} consists of the following two steps: (i) show that, except for the upper bound in (\ref{Ineq:upperbound je}), all terms in the lower bound in (\ref{Ineq:lowerbound je}) are super-exponentially small; and (ii) prove that the difference between the upper bound in (\ref{Ineq:upperbound je}) and $\E\{\chi(A_u)\}$ is also super-exponentially small. The ideas for proving Theorems \ref{Thm:MEC approximation je2} and \ref{Thm:MEC approximation} are similar.

\section{Estimation of super-exponentially small terms in the lower bound}\label{sec:small}
\subsection{Auxiliary results on multivariate Gaussian tails}
\begin{lemma}\label{Lem:estimating joint tail for 3 rvs} Let $\{(\xi_1(x_1), \xi_2(x_2), \xi_3(x_3)): (x_1, x_2, x_3)\in D_1\times D_2 \times D_3\}$
	be an $\R^3$-valued, $C^2$, centered, unit-variance, non-degenerate Gaussian vector field, where $D_i$, $i=1,2,3$, 
	are compact sets in $\R^N$. Let
	$R_{ij}=\sup_{x_i\in D_i, x_j\in D_j}\E \{\xi_i(x_i)\xi_j(x_j)\}$,
where $i,j=1,2,3$ and $i<j$. If $R_{12}\le \min\{R_{13}, R_{23}\}$, then there exists a constant $\alpha>0$ such that for  every integer  $m\ge 0$, 
as $u\to \infty$,
	\begin{equation}\label{Eq:triple tail}
	\begin{split}
	&\sup_{x_1\in D_1, x_2\in D_2, x_3\in D_3} \E\{|\xi_1(x_1)\xi_2(x_2)\xi_3(x_3)|^{m} \mathbbm{1}_{\{\xi_1(x_1) \geq u, \xi_2(x_2) \geq u, \xi_3(x_3) \geq u\}} \}\\
	&\quad = o\bigg(\exp\bigg\{-\alpha u^2-\frac{u^2}{1+R_{12}} \bigg\}\bigg).
	\end{split}
	\end{equation}
	
\end{lemma}
\begin{proof} Due to the exponential decay of Gaussian tails, it suffices to prove that there exists $\alpha'>0$ such that as $u\to \infty$,
	\begin{equation}\label{Eq:triple tail 2}
		\begin{split}
			&\sup_{x_1\in D_1, x_2\in D_2, x_3\in D_3} \P\{\xi_1(x_1) \geq u, \xi_2(x_2) \geq u, \xi_3(x_3) \geq u\} \} 
			= o\left(\exp\left\{-\alpha' u^2-\frac{u^2}{1+R_{12}} \right\}\right).
		\end{split}
	\end{equation}
	Note that,
	\[
	\P\{\xi_1(x_1) \geq u, \xi_2(x_2) \geq u, \xi_3(x_3) \geq u\} \} \le \P\{(\xi_1(x_1) +\xi_2(x_2))/2 \geq u, \xi_3(x_3) \geq u\} \},
	\]
	where $(\xi_1(x_1) +\xi_2(x_2))/2$ is a centered Gaussian variable with variance bounded by
	\[
	\sup_{x_1\in D_1, x_2\in D_2}{\rm Var}((\xi_1(x_1) +\xi_2(x_2))/2)= \frac{1+R_{12}}{2}.
	\]
It is known that (see for example \citet{Tong1990}), for a centered nondegenerate bivariate Gaussian vector $(Z_1, Z_2)$ with 
${\rm Var}(Z_1) = \sigma^2$, there exists $\alpha'>0$ such that as $u\to \infty$,
\[
\P\{Z_1 \geq u, Z_2 \geq u\} = o\left(\exp\left\{-\alpha' u^2-\frac{u^2}{2\sigma^2} \right\}\right).
\]
Combining these yields \eqref{Eq:triple tail 2} and hence \eqref{Eq:triple tail}.
\end{proof}

\begin{lemma}\label{Lem:estimating conditional joint tail} Let $\big\{(\xi_1(x_1),  \ldots, \xi_n(x_n): x_i\in D_i, i=1, \ldots, n\big\}$
	be an $\R^n$-valued, $C^2$, centered, unit-variance, non-degenerate Gaussian vector vector, where $D_1$, \ldots, $D_n$ ($n\ge 3$) are compact sets in $\R^N$. Let 
	$R_{12}=\sup_{x_1\in D_1, x_2\in D_2}\E\{\xi_1(x_1) \xi_2(x_2)\}$.
	If
	\begin{equation}\label{Condition:coditional joint tail}
	\begin{split}
	\{&(x_1, \ldots, x_n)\in D_1\times \cdots \times D_n:  \\
	&\quad \E\{\xi_1(x_1) \xi_2(x_2)\}=R_{12}, \ \E \{(\xi_1(x_1)+\xi_2(x_2))\xi_i(x_i)\}= 0, \ \forall i=3,\ldots, n\} = \emptyset,
	\end{split}
	\end{equation}
	then there exists $\alpha>0$ such that as $u\to \infty$,
	\begin{equation*}
	\begin{split}
	&\sup_{x_i\in D_i, i=1,\ldots, n } \E\{|\xi_1(x_1)\xi_2(x_2)|^{m} \mathbbm{1}_{\{\xi_1(x_1) \geq u, \xi_2(x_2) \geq u\}}| \xi_3(x_3)=\cdots =\xi_n(x_n)=0 \}\\
	&\quad = o\bigg(\exp\bigg\{-\alpha u^2-\frac{u^2}{1+R_{12}} \bigg\}\bigg),
	\end{split}
	\end{equation*}
where $m\ge 0$ is any fixed integer.
\end{lemma}
\begin{proof} Let $\ol{\xi}(x_1,x_2)=[\xi_1(x_1) + \xi_2(x_2)]/2$. Then
	\begin{equation*}
	\begin{split}
	\E&\{|\xi_1(x_1)\xi_2(x_2)|^{m} \mathbbm{1}_{\{\xi_1(x_1) \geq u, \xi_2(x_2) \geq u\}}| \xi_3(x_3)=\cdots =\xi_n(x_n)=0 \} \\
	&\leq \E\{\ol{\xi}(x_1,x_2)^{2m} \mathbbm{1}_{\{\ol{\xi}(x_1,x_2) \geq u\}}| \xi_3(x_3)=\cdots =\xi_n(x_n)=0 \}.
	\end{split}
	\end{equation*}
	Note that $(\ol{\xi}(x_1,x_2) | \xi_3(x_3)=\cdots =\xi_n(x_n)=0)$ is a centered Gaussian variable with variance
	\begin{equation*}
	\begin{split}
	{\rm Var}(\ol{\xi}(x_1,x_2)|\xi_3(x_3)=\cdots =\xi_n(x_n)=0) &\le {\rm Var}(\ol{\xi}(x_1,x_2))= \frac{1+\E\{\xi_1(x_1) \xi_2(x_2)\}}{2}\le \frac{1+R_{12}}{2},
	\end{split}
	\end{equation*}
	where the first inequality becomes equality if and only if $\ol{\xi}(x_1,x_2)$ is independent of each $\xi_i(x_i)$, $i \ge 3$. 
	The desired result follows from the continuity of the conditional variance in $x_i$ and the compactness of $D_i$, $i=1,\ldots, n$.
\end{proof}

\subsection{Non-adjacent faces}
For two sets $D, D' \subset \R^N$, let $d(D,D')=\inf\{\|t-t'\|: t\in D, t'\in D'\}$ denote their distance. The following 
result shows that the last two sums involving the joint moment of two non-adjacent faces in \eqref{Ineq:lowerbound je} 
are super-exponentially small.
\begin{lemma}\label{Lem:cross terms disjoint sets} Let $\{(X(t),Y(s)): t\in T, s\in S\}$ be an $\R^2$-valued, centered, 
unit-variance Gaussian vector field satisfying $({\bf H}1)$ and $({\bf H}2)$. Then there exists $\alpha>0$ such that as $u\to \infty$,
	\begin{equation}\label{Eq:crossing terms disjoint sets}
	\begin{split}
	\E\{M_u (X,K) M_u (X,K') M_u (Y,L)\} &= o\Big(\exp \Big\{ -\frac{u^2}{1+R} -\alpha u^2 \Big\}\Big),\\
	\E\{M_u (X,K) M_u (Y,L) M_u (Y,L')\} &= o\Big(\exp \Big\{ -\frac{u^2}{1+R} -\alpha u^2 \Big\}\Big),
	\end{split}
	\end{equation}
	where $K$ and $K'$ are different faces of $T$ with $d(K,K')>0$, $L$ and $L'$ are different faces of $S$ with $d(L,L')>0$.
\end{lemma}
\begin{proof} We only prove the first line in (\ref{Eq:crossing terms disjoint sets}), since the proof for the second line is similar. 
Consider first the case when ${\rm dim}(K) =k \geq 1$, ${\rm dim}(K') =k' \geq 1$ and ${\rm dim}(L) =l \geq 1$. By the Kac-Rice 
metatheorem for high moments \cite{Adler:2007},
	\begin{equation}\label{Eq:disjoint faces}
	\begin{split}
	&\quad \E\{M_u (X,K) M_u (X,K') M_u (Y,L)\}\\
	&= \int_{K} dt \int_{K'} dt'\int_{L} ds \, \E\big \{ |{\rm det} \nabla^2 X_{|K}(t) | |{\rm det} \nabla^2 X_{|K'}(t') | |{\rm det} \nabla^2 Y_{|L}(s) |\\
	&\quad \times \mathbbm{1}_{\{X(t)\geq u, X(t')\geq u, Y(s)\geq u\}} \mathbbm{1}_{\{\nabla^2 X_{|K}(t) \prec 0, \, \nabla^2 X_{|K'}(t') \prec 0, \, \nabla^2 Y_{|L}(s) \prec 0\}} \big|\\
	&\quad  \nabla X_{|K}(t)=0, \nabla X_{|K'}(t')=0, \nabla Y_{|L}(s)=0 \big\}p_{\nabla X_{|K}(t), \nabla X_{|K'} (t'), \nabla Y_{|L}(s)}(0,0,0) \\
	&\leq \int_{K} dt \int_{K'} dt' \int_{L} ds \int_u^\infty dx \int_u^\infty dx' \int_u^\infty dy \, p_{X(t), X(t'), Y(s)}(x,x',y)\\
	&\quad \times \E\big \{ |{\rm det} \nabla^2 X_{|K}(t) | |{\rm det} \nabla^2 X_{|K'}(t') ||{\rm det} \nabla^2 Y_{|L}(s)|\big |\\
	&\quad X(t)=x, X(t')=x', Y(s)=y, \nabla X_{|K}(t)=0, \nabla X_{|K'}(t')=0, \nabla Y_{|L}(s)=0 \} \\
	&\quad \times p_{\nabla X_{|K}(t), \nabla X_{|K'} (t'), \nabla Y_{|L}(s)}(0,0,0|X(t)=x, X(t')=x', Y(s)=y).
	\end{split}
	\end{equation}
	Notice that the following two inequalities hold: for constants $a_{i_1}$, $b_{i_2}$ and $c_{i_3}$,
	\begin{equation*}
	\begin{split}
	&\prod_{i_1=1}^k |a_{i_1}| \prod_{i_2=1}^{k'} |b_{i_2}| \prod_{i_3=1}^{l} |c_{i_3}|\leq \frac{\sum_{i_1=1}^k |a_{i_1}|^{k+k'+l} 
	+ \sum_{i_2=1}^{k'} |b_{i_2}|^{k+k'+l} + \sum_{i_3=1}^{l} |c_{i_3}|^{k+k'+l}}{k+k'+l};
	\end{split}
	\end{equation*}
	and for any Gaussian variable $\xi$ and positive integer $m$, by Jensen's inequality,
	\begin{equation*}
	\begin{split}
	\E |\xi|^m \leq \E (|\E\xi|+|\xi-\E\xi|)^m &\leq 2^{m-1} (|\E\xi|^m + \E |\xi-\E\xi|^m)= 2^{m-1} (|\E\xi|^m + B_m({\rm Var}(\xi))^{m/2}),
	\end{split}
	\end{equation*}
	where $B_m$ is some constant depending only on $m$. Combining these two inequalities with the well-known conditional 
	formula for Gaussian variables, we obtain that there exist positive constants $C_1$ and $N_1$ such that for large $x$, $x'$ and $y$,
	\begin{equation}\label{Eq:disjoint faces 2}
	\begin{split}
	\sup_{t\in K, t'\in K', s\in L}&\E\big \{ |{\rm det} \nabla^2 X_{|K}(t) | |{\rm det} \nabla^2 X_{|K'}(t') | |{\rm det} \nabla^2 Y_{|L}(s) | \big| X(t)=x, X(t')=x',\\
	& Y(s)=y, \nabla X_{|K}(t)=0, \nabla X_{|K'}(t')=0,  \nabla Y_{|L}(s)=0 \big\} \leq C_1+(xx'y)^{N_1}.
	\end{split}
	\end{equation}
	Further, there exists $C_2>0$ such that
	\begin{equation}\label{Eq:disjoint faces 3}
	\begin{split}
	&\quad \sup_{t\in K, t'\in K', s\in L} p_{\nabla X_{|K}(t), \nabla X_{|K'} (t'), \nabla Y_{|L}(s)}(0,0,0|X(t)=x, X(t')=x', Y(s)=y)\\
	& \leq  \sup_{t\in K, t'\in K', s\in L}(2\pi)^{-(k+k'+l)/2}[{\rm det Cov} (\nabla X_{|K}(t), \nabla X_{|K'} (t'), \nabla Y_{|L}(s) | \\
	& \qquad \qquad\qquad\qquad\qquad\qquad\qquad X(t)=x, X(t')=x', Y(s)=y)]^{-1/2} \leq  C_2.
	\end{split}
	\end{equation}
Plugging \eqref{Eq:disjoint faces 2} and \eqref{Eq:disjoint faces 3} into \eqref{Eq:disjoint faces}, we obtain that there exists $C_3={\rm Vol}(K){\rm Vol}(K'){\rm Vol}(L)$ such that
\begin{equation}\label{Eq:disjoint faces 4}
	\begin{split}
&\E\{M_u (X,K) M_u (X,K') M_u (Y,L)\} \\
&\quad \le C_3C_2 \sup_{t\in K, t'\in K', s\in L} \E\{(C_1+|X(t)X(t')Y(s)|^{N_1}) \mathbbm{1}_{\{X(t) \geq u, X(t') \geq u, Y(s) \geq u\}} \}.
\end{split}
\end{equation}
The desired result then follows from Lemma \ref{Lem:estimating joint tail for 3 rvs}. The case when one of the dimensions 
of $K$, $K'$ and $L$ is zero can be proved similarly.
\end{proof}

\subsection{Factorial moments}

The following result shows that the factorial moments in \eqref{Ineq:lowerbound je} are super-exponentially small.
\begin{lemma}\label{Lem:factorial moments je} Let $\{(X(t),Y(s)): t\in T, s\in S\}$ be an $\R^2$-valued, centered, 
unit-variance Gaussian vector field satisfying $({\bf H}1)$, $({\bf H}2)$ and $({\bf H}3)$. Then there exists a
constant $\alpha>0$ such that for all $K\in \partial_k T$ and $L\in \partial_l T$ with $k, l\ge 0$, as $u\to \infty$,
	\begin{equation}\label{Eq:factorial moments je}
		\begin{split}
			\E \{M_u (X,K)[M_u (X,K)-1]M_u (Y,L) \} &= o\Big(\exp \Big\{ -\frac{u^2}{1+R} -\alpha u^2 \Big\}\Big),\\
			\E \{M_u (X,K)M_u (Y,L)[M_u (Y,L)-1] \} &= o\Big(\exp \Big\{ -\frac{u^2}{1+R} -\alpha u^2 \Big\}\Big).
		\end{split}
	\end{equation}
\end{lemma}
\begin{proof} We only prove the first line in (\ref{Eq:factorial moments je}), since the proof for the second line is similar. 
Note that, if $k=0$, then $M_u (X,K)[M_u (X,K)-1]\equiv 0$ and hence the desired result holds. Without loss of generality, 
we assume $k\ge 1$ or even $k=N$ for simplifying notations. We first focus on the estimation when $K$ is replaced by 
a small $N$-dimensional subset $J\subset K$. 
	
	\textbf{Case (i): $\bm{l=0}$.} The face $L$ becomes a single point, say $L=\{s\}$. Applying the Kac-Rice metatheorem 
	for high moments \cite{Adler:2007}, we have the following upper bounds (removing one restriction on $u$ and 
	another restriction on the negative definiteness of Hessian matrices), 
	\begin{equation}\label{Eq:factorial 1}
		\begin{split}
			&\quad \E \{M_u (X,J)[M_u (X,J)-1]M_u (Y,L) \}\\
			&\le \int_J dt \int_J dt'\, \E\big \{ |{\rm det} \nabla^2 X(t) | |{\rm det} \nabla^2 X(t') | \mathbbm{1}_{\{X(t)\geq u, Y(s)\geq u\}}\big|\nabla X(t)= \nabla X(t')=0\big\}\\
			&\quad \times p_{\nabla X(t), \nabla X(t')}(0,0) \\
			&\leq \int_J dt \int_J dt'\int_u^\infty dx \, p_{\frac{X(t)+ Y(s)}{2}}(x|\nabla X(t)=\nabla X(t')=0)p_{\nabla X(t), \nabla X(t')}(0,0)\\
			&\quad \times \E\big \{ |{\rm det} \nabla^2 X(t) | |{\rm det} \nabla^2 X(t') | \big|(X(t)+ Y(s))/2=x, \nabla X(t)=\nabla X(t')=0\big\},
		\end{split}
	\end{equation}
	where the last inequality is due to the fact $\mathbbm{1}_{\{X(t)\geq u, Y(s)\geq u\}}\le \mathbbm{1}_{\{[X(t)+ Y(s)]/2\geq u\}}$. 
	Following the same arguments for proving Lemma 3 in \citet{Piterbarg96}, we obtain from (\ref{Eq:factorial 1}) that, for any $\ep>0$, 
	there exists $\delta>0$ such that for $J$ with ${\rm diam}(J)= \sup_{t,t'\in J}\|t-t'\| \leq \delta$ and $u$ large enough,
	\begin{equation}\label{Eq:factorial_moment 2}
		\E \{M_u (X,J)[M_u (X,J)-1]M_u (Y,L) \} \leq \exp\left\{ -\frac{u^2}{2\beta(J,L) + \ep} \right\},
	\end{equation}
where 
\begin{equation}\label{Eq:beta}
	\begin{split}
		\beta(J,L)=\sup_{t\in J, s\in L, e\in \mathbb{S}^{N-1}} {\rm Var}((X(t)+ Y(s))/2 |\nabla X(t)=0, \nabla^2 X(t)e=0),
	\end{split}
\end{equation}
and  $\mathbb{S}^{N-1}$ the $(N-1)$-dimensional unit sphere in $\R^N$.
	
	\textbf{Case (ii): $\bm{l\ge 1}$. } To simplify the notations, without loss of generality, we assume $l=N$. Applying again the 
	Kac-Rice metatheorem for high moments, we have the following upper bounds, 
	\begin{equation}\label{Eq:factorial 2}
		\begin{split}
			&\quad \E \{M_u (X,J)[M_u (X,J)-1]M_u (Y,L) \}\\
			&\le \int_J dt \int_J dt'\int_L ds \, \E\big \{ |{\rm det} \nabla^2 X(t) | |{\rm det} \nabla^2 X(t') | |{\rm det} \nabla^2 Y(s) |\mathbbm{1}_{\{X(t)\geq u, Y(s)\geq u\}}\big|\\
			&\quad \nabla X(t)=0, \nabla X(t')=0, \nabla Y(s)=0 \big\}p_{\nabla X(t), \nabla X(t'), \nabla Y(s)}(0,0,0) \\
			&\leq \int_J dt \int_J dt'\int_L ds \int_u^\infty dx \, p_{\frac{X(t)+ Y(s)}{2}}(x|\nabla X(t)=0, \nabla X(t')=0, \nabla Y(s)=0)\\
			&\quad \times \E\big \{ |{\rm det} \nabla^2 X(t) | |{\rm det} \nabla^2 X(t') | |{\rm det} \nabla^2 Y(s) |\big|[X(t)+ Y(s)]/2=x, \\
			&\quad \nabla X(t)=0, \nabla X(t')=0, \nabla Y(s)=0 \big\} p_{\nabla X(t), \nabla X(t'), \nabla Y(s)}(0,0,0).
		\end{split}
	\end{equation}
Comparing \eqref{Eq:factorial 2} with \eqref{Eq:factorial 1}, the only essential difference is the additional effect of $\nabla Y(s)=0$, 
which however will not affect the desired super-exponentially small estimation since $(X, Y)$ is nondegenerate under the condition 
({\bf H}2). Therefore, similarly to (\ref{Eq:factorial_moment 2}), we have that, for any $\ep>0$, there exists $\delta>0$ such that
 for $J$ with ${\rm diam}(J) \leq \delta$ and $u$ large enough,
\begin{equation}\label{Eq:factorial_moment}
	\E \{M_u (X,J)[M_u (X,J)-1]M_u (Y,L) \} \leq \exp\left\{ -\frac{u^2}{2\gamma(J,L) + \ep} \right\}\leq \exp\left\{ -\frac{u^2}{2\beta(J,L) + \ep} \right\},
\end{equation}
where
\begin{equation*}
	\begin{split}
		\gamma(J,L)&=\sup_{t\in J, s\in L, e\in \mathbb{S}^{N-1}} {\rm Var}((X(t)+ Y(s))/2 | \nabla X(t)= \nabla Y(s)=\nabla^2 X(t)e=0)\le \beta(J,L).
	\end{split}
\end{equation*}

	The set $K$ may be covered by congruent cubes $J_i$ with disjoint interiors, edges parallel to coordinate axes 
	and sizes small enough such that ${\rm diam}(J_i\cup J_j)\le \delta$ for any two neighboring cubes $J_i$ and 
	$J_j$ (i.e., $d(J_i, J_j)=0$). Then
	\begin{equation}\label{Eq:decompose factorial moment je}
	\begin{split}
	&\quad \E \{M_u (X,K)[M_u (X,K)-1]M_u (Y,L) \}\\
	&\leq \E \Big\{\Big(\sum_i M_u (X,J_i)\Big)\Big[ \sum_j M_u (X,J_j)-1 \Big]M_u (Y,L) \Big\}\\
	&=\E \Big\{\Big(\sum_i M_u (X,J_i)\sum_j M_u (X,J_j) - \sum_i M_u (X,J_i) \Big)M_u (Y,L) \Big\}\\
	&=\sum_i \E \{M_u (X,J_i)^2M_u (Y,L)\} + \sum_{i\neq j} \E \{M_u (X,J_i)M_u (X,J_j)M_u (Y,L)\}\\
	&\quad - \sum_i \E \{M_u (X,J_i)M_u (Y,L)\} \\
	&= \sum_i \E \{M_u (X,J_i)[M_u (X,J_i)-1]M_u (Y,L)\}+ \sum_{i\neq j} \E \{M_u (X,J_i)M_u (X,J_j)M_u (Y,L)\}.
	\end{split}
	\end{equation}
	By Lemma \ref{Lem:cross terms disjoint sets}, there exists $\alpha'>0$ such that for $u$ large enough,
	\begin{equation}\label{Eq:high moments for nonneighboring sets}
	\begin{split}
	\sum_{i\neq j:\, d(J_i, J_j)>0} \E \{M_u (X,J_i)M_u (X,J_j)M_u (Y,L)\} \leq \exp\Big\{ -\frac{u^2}{1+R} -\alpha' u^2\Big\}.
	\end{split}
	\end{equation}
	If $J_i$ and $J_j$ are neighboring, i.e., $d(J_i, J_j)=0$, we have
	\begin{equation}\label{Eq:neighbor}
	\begin{split}
	&\quad \E \{M_u(X, J_i \cup J_{j})[M_u(X, J_i \cup J_{j})-1]M_u (Y,L)\} \\
	&= \E \{ [M_u(X, J_i)+ M_u(X, J_{j})][M_u(X, J_i)+ M_u(X, J_{j})-1]M_u (Y,L)\}\\
	&= 2 \E\{M_u(X, J_i)M_u(X, J_{j})M_u (Y,L)\} + \E \{M_u(X,J_i)[M_u(X,J_i)-1]M_u (Y,L)\}\\
	&\quad  + \E \{M_u(X, J_{j})[M_u(X, J_{j})-1]M_u (Y,L)\}.
	\end{split}
	\end{equation}
	Applying \eqref{Eq:factorial_moment 2} and \eqref{Eq:factorial_moment} to the second last sum in 
	\eqref{Eq:decompose factorial moment je} and \eqref{Eq:neighbor}, we see that for any $\ep>0$ and $u$ large enough,
	\begin{equation}\label{Eq:small factorial moments and neighboring sets}
	\begin{split}
	&\sum_i \E \{M_u (X,J_i)[M_u (X,J_i)-1]M_u (Y,L)\} \\
	&+  \sum_{i\neq j:\, d(J_i, J_j)=0} \E \{M_u (X,J_i)M_u (X,J_j)M_u (Y,L)\} \leq \exp\Big\{ -\frac{u^2}{2\beta(K,L)+\ep} \Big\},
	\end{split}
	\end{equation}
	where $\beta(K,L)$ is defined in \eqref{Eq:beta} with $J$ replaced by $K$. It is evident that
	\[
	\beta(K,L) \le \sup_{t\in K, s\in L} {\rm Var}((X(t)+ Y(s))/2) = (1+R)/2.
	\] 
	Moreover, we will how below that 
	\begin{equation}\label{Eq:comparing two sup}
	\begin{split}
	\beta(K,L)<(1+R)/2.
	\end{split}
	\end{equation}
	By the definition, if $\beta(K,L)= (1+R)/2$, then there exist $(t,s)\in \bar{K}\times \bar{L}$ and $e\in \mathbb{S}^{N-1}$ such that
	\begin{equation}\label{Eq:conditional var attains max}
	\begin{split}
	{\rm Var}((X(t)+ Y(s))/2 |\nabla X(t)=0, \nabla^2 X(t)e=0) = (1+R)/2,
	\end{split}
	\end{equation}
	implying $r(t,s)=R$ and $\E\{[X(t)+Y(s)]\nabla X(t)\}=\E\{Y(s)\nabla X(t)\}=0$. By Proposition \ref{prop:H3}, 
	$\E\{Y(s) \nabla^2 X(t)\}\preceq 0$. Since $X(t)$ has unit variance, $\E\{X(t) \nabla^2 X(t)\}=-{\rm Cov}(\nabla X(t))\prec 0$. 
	Therefore,  $\E\{[X(t)+Y(s)] \nabla^2 X(t)e\} \neq 0$ for all $e\in \mathbb{S}^{N-1}$. This contradicts (\ref{Eq:conditional var attains max}) 
	and hence (\ref{Eq:comparing two sup}) holds. Applying this fact and plugging (\ref{Eq:high moments for nonneighboring sets}) 
	and (\ref{Eq:small factorial moments and neighboring sets}) into (\ref{Eq:decompose factorial moment je}), we finish the proof.
\end{proof}

\subsection{Adjacent faces}
The following result shows that the last two sums involving the joint moment of two adjacent faces in \eqref{Ineq:lowerbound je} are super-exponentially small.
\begin{lemma}\label{Lem:cross terms je} Let $\{(X(t),Y(s)): t\in T, s\in S\}$ be an $\R^2$-valued, centered, unit-variance Gaussian vector field satisfying $({\bf H}1)$, $({\bf H}2)$ and $({\bf H}3)$. Then there exists $\alpha>0$ such that as $u\to \infty$,
	\begin{equation}\label{Eq:crossing terms je}
	\begin{split}
	\E\{M_u^E (X,K) M_u^E (X,K') M_u^E (Y,L)\} &= o\Big(\exp \Big\{ -\frac{u^2}{1+R} -\alpha u^2 \Big\}\Big),\\
	\E\{M_u^E (X,K) M_u^E (Y,L) M_u^E (Y,L')\} &= o\Big(\exp \Big\{ -\frac{u^2}{1+R} -\alpha u^2 \Big\}\Big),
	\end{split}
	\end{equation}
where $K$ and $K'$ are different faces of $T$ with $d(K,K')=0$, $L$ and $L'$ are different faces of $S$ with $d(L,L')=0$.
\end{lemma}
\begin{proof} We only prove the first line in (\ref{Eq:crossing terms je}), since the proof for the second line is the same. Let $I:= \bar{K}\cap \bar{K'}$, which is nonempty since $d(K,K')=0$. Without loss of generality, assume
	\begin{equation*}
	\begin{split}
	\sigma(K) &= \{1, \ldots, m, m+1, \ldots, k\},\\
	\sigma(K')&= \{1, \ldots, m, k+1, \ldots, k+k'-m\},\\
	\sigma(L) &= \{1, \ldots, l\},
	\end{split}
	\end{equation*}
	where $0 \leq m \leq k \leq k' \leq N$ and $k'\geq 1$. If $k=0$, we consider $\sigma(K)
	= \emptyset$ by convention. Under such assumption, $K\in \partial_k T$, $K'\in \partial_{k'} T$, $\text{dim} (I) =m$ and $L\in \partial_l S$. We assume also that all elements in $\ep(K)$ and $\ep(K')$ are 1.
	
	We first consider the case when $k\geq 1$ and $l\ge 1$. By the Kac-Rice metatheorem,
	\begin{equation}\label{Eq:cross term je}
	\begin{split}
	&\quad \E\{M_u^E (X,K) M_u^E (X,K') M_u^E (Y,L)\}\\
	&\leq \int_{K} dt\int_{K'} dt'\int_{L} ds \int_u^\infty dx \int_u^\infty dx' \int_u^\infty dy  \int_0^\infty dz_{k+1} \cdots \int_0^\infty dz_{k+k'-m}\\
	&\quad \int_0^\infty dw_{m+1} \cdots \int_0^\infty dw_{k} \E\big \{ |\text{det} \nabla^2 X_{|K}(t) ||\text{det} \nabla^2 X_{|K'}(t') ||\text{det} \nabla^2 Y_{|L}(s) | \big|\\
	&\quad X(t)=x, X(t')=x', Y(s)=y,\\
	&\quad \nabla X_{|K}(t)=0,  X_{k+1}(t)=z_{k+1}, \ldots, X_{k+k'-m}(t)=z_{k+k'-m}, \\
	& \quad \nabla X_{|K'}(t')=0, X_{m+1}(t')=w_{m+1}, \ldots, X_{k}(t')=w_{k}, \nabla Y_{|L}(s)=0 \big\}\\
	& \quad \times p_{t,t',s}(x,x',y,0, z_{k+1}, \ldots,z_{k+k'-m}, 0, w_{m+1}, \ldots, w_{k}, 0)\\
	&:= \int \int \int_{K\times K'\times L} A(t,t',s)\,dtdt'ds,
	\end{split}
	\end{equation}
	where $p_{t,t',s}(x,x',y,0, z_{k+1}, \ldots,z_{k+k'-m}, 0, w_{m+1}, \ldots, w_{k}, 0)$ is the density of
	\begin{equation*}
	\begin{split}
	&\big(X(t),X(t'),Y(s),\nabla X_{|K}(t), X_{k+1}(t),\ldots, X_{k+k'-m}(t), \\
	&\quad \nabla X_{|K'}(t'), X_{m+1}(t'), \ldots, X_{k}(t'), \nabla Y_{|L}(s)\big)
	\end{split}
	\end{equation*}
	evaluated at $(x,x',y,0, z_{k+1}, \ldots,z_{k+k'-m}, 0, w_{m+1}, \ldots, w_{k}, 0)$. We define 
	\begin{equation}\label{Eq:M0}
	\begin{split}
	\mathcal{M}_0:=\{(t,s)\in I \times \bar{L}:\, &r(t,s)=R, \, \E\{X_i(t)Y(s)\}=\E\{X(t)Y_j(s)\}=0, \\
	&\forall i=1,\ldots, k+k'-m, \, j=1, \ldots, l\},
	\end{split}
	\end{equation}
	and distinguish two cases for $\mathcal{M}_0$ in discussions below.
	
	\textbf{Case (i): $\bm{\mathcal{M}_0 = \emptyset}$.} Since $I$ is a compact set, by the uniform continuity of conditional 
	variance, there exist constants $\ep_1, \delta_1>0$ such that
	\begin{equation}\label{Eq:super-exponentially small}
	\begin{split}
	&\sup_{t\in B(I, \delta_1), \, t'\in B'(I, \delta_1), \, s\in L} {\rm Var} ( [X(t)+Y(s)]/2 |\nabla X_{|K}(t), \nabla X_{|K'} (t'), \nabla Y_{|L}(s))\leq \frac{1+R}{2} - \ep_1,
	\end{split}
	\end{equation}
	where $B(I, \delta_1)=\{t\in K: d(t, I)\le \delta_1\}$ and $B'(I, \delta_1)=\{t\in K': d(t, I)\le \delta_1\}$. Partitioning $K\times K'$ into $B(I, \delta_1) \times B'(I, \delta_1)$ and $(K\times K')\backslash (B(I, \delta_1) \times B'(I, \delta_1))$, and applying the Kac-Rice formula, we obtain
	\begin{equation}\label{Eq:Mu3}
	\begin{split}
	&\quad \E\{M_u (X,K) M_u (X,K') M_u (Y,L)\}\\
	&\leq \int_{(K\times K')\backslash (B(I, \delta_1) \times B'(I, \delta_1))} dt dt' \int_{L} ds \, p_{\nabla X_{|K}(t), \nabla X_{|K'} (t'), \nabla Y_{|L}(s)}(0,0,0)\\
	&\quad \times \E \big\{ |{\rm det} \nabla^2 X_{|K}(t) | |{\rm det} \nabla^2 X_{|K'}(t') | |{\rm det} \nabla^2 Y_{|L}(s) |\mathbbm{1}_{\{X(t)\geq u, X(t')\geq u, Y(s)\geq u\}}\big| \\
	&\qquad \nabla X_{|K}(t)=0, \nabla X_{|K'}(t')=0, \nabla Y_{|L}(s)=0 \big\} \\
	&+ \int_{B(I, \delta_1) \times B'(I, \delta_1)} dt dt' \int_{L} ds \, p_{\nabla X_{|K}(t), \nabla X_{|K'} (t'), \nabla Y_{|L}(s)}(0,0,0)\\
	&\quad \times \E \big\{ |{\rm det} \nabla^2 X_{|K}(t) | |{\rm det} \nabla^2 X_{|K'}(t') ||{\rm det} \nabla^2 Y_{|L}(s) | \mathbbm{1}_{\{X(t)\geq u, Y(s)\geq u\}}\big| \\
	&\qquad \nabla X_{|K}(t)=0, \nabla X_{|K'}(t')=0, \nabla Y_{|L}(s)=0 \big\} \\
	&:=I_1 + I_2.
	\end{split}
	\end{equation}
	Note that
	\begin{equation*}
	\begin{split}
	(K\times K') \backslash (B(I, \delta_1) \times B'(I, \delta_1)) &= \Big( (K\backslash B(I, \delta_1)) \times B'(I, \delta_1) \Big) \bigcup \Big( B(I, \delta_1)\times (K\backslash B(I, \delta_1)) \Big) \\
	&  \quad \bigcup  \Big( (K\backslash B(I, \delta_1))\times (K\backslash B(I, \delta_1)) \Big).
	\end{split}
	\end{equation*}
	where each product on the right hand side consists of two sets with a positive distance. It then follows from Lemma \ref{Lem:cross terms disjoint sets} that $I_1$ is super-exponentially small. On the other hand, since $\mathbbm{1}_{\{X(t)\geq u, Y(s)\geq u\}}\le \mathbbm{1}_{\{[X(t)+ Y(s)]/2\geq u\}}$, one has
	\begin{equation}\label{Eq:I2}
		\begin{split}
			I_2&\le\int_{B(I, \delta_1) \times B'(I, \delta_1)} dt dt' \int_{L} ds\int_u^\infty dx\,p_{\frac{X(t)+Y(s)}{2}}(x|\nabla X_{|K}(t)=\nabla X_{|K'}(t')=\nabla Y_{|L}(s)=0) \\
			&\quad \times \E \big\{ |{\rm det} \nabla^2 X_{|K}(t) | |{\rm det} \nabla^2 X_{|K'}(t') ||{\rm det} \nabla^2 Y_{|L}(s) | \big| [X(t)+ Y(s)]/2=x,\\
			&\qquad \nabla X_{|K}(t)=0, \nabla X_{|K'}(t')=0, \nabla Y_{|L}(s)=0 \big\}p_{\nabla X_{|K}(t), \nabla X_{|K'} (t'), \nabla Y_{|L}(s)}(0,0,0).
		\end{split}
	\end{equation}
	 Combining this with \eqref{Eq:super-exponentially small}, we obtain that $I_2$ and hence $\E\{M_u^E (X,K) M_u^E (X,K') M_u^E (Y,L)\}$ are super-exponentially small.

	\textbf{Case (ii): $\bm{\mathcal{M}_0 \neq \emptyset}$.} Let
	\begin{equation*}
	\begin{split}
	B(\mathcal{M}_0,\delta_2):=\{(t,t',s)\in K\times K' \times L: d((t,s),\mathcal{M}_0)\vee d((t',s),\mathcal{M}_0)\le \delta_2  \},
	\end{split}
	\end{equation*}
	where $\delta_2$ is a small positive number to be specified. Note that, by the definitions of $\mathcal{M}_0$ and 
	$B(\mathcal{M}_0,\delta_2)$, there exists $\ep_2>0$ such that
	\begin{equation}\label{Eq:delta2}
	\begin{split}
	\sup_{(t,t',s)\in (K\times K'\times L) \backslash B(\mathcal{M}_0,\delta_2)} &{\rm Var} ( [X(t)+Y(s)]/2 |\nabla X_{|K}(t), \nabla X_{|K'} (t'), \nabla Y_{|L}(s))\\
	&\leq \frac{1+R}{2} - \ep_2.
	\end{split}
	\end{equation}
	Similarly to \eqref{Eq:I2}, we obtain that $\int_{(K\times K'\times L) \backslash B(\mathcal{M}_0,\delta_2)}A(t,t',s)dtdt'ds$ is 
	super-exponentially small. It suffices to show below that $\int _{B(\mathcal{M}_0,\delta_2)} A(t,t',s)\,dtdt'ds$ is super-exponentially small.
	
	Due to $({\bf H}3)$ and Proposition \ref{prop:H3}, we can choose $\delta_2$ small enough such that for all $(t,t',s)\in B(\mathcal{M}_0,\delta_2)$,
	\begin{equation*}
	\begin{split}
	\La_{K\cup K'}(t,s)&:=-\E\{[X(t)+Y(s)]\nabla^2 X_{|K\cup K'}(t)\}=-(\E\{[X(t)+Y(s)] X_{ij}(t)\})_{i,j=1,\ldots, k+k'-m}
	\end{split}
	\end{equation*}
	are positive definite. Let $\{e_1, e_2, \ldots, e_N\}$ be the standard orthonormal basis of $\R^N$. For $t\in K$, $t'\in K'$ and $s\in L$, let $e_{t, t'}=(t'-t)/\|t'-t\|$ and $\alpha_i(t, t',s)= \l e_i, \La_{K\cup K'}(t,s)e_{t,t'}\r$. Then
	\begin{equation}\label{Eq:orthnormal basis decomp je}
	\La_{K\cup K'}(t,s)e_{t,t'}=\sum_{i=1}^N \l e_i, \La_{K\cup K'}(t,s)e_{t,t'}\r e_i = \sum_{i=1}^N  \alpha_i(t, t', s) e_i
	\end{equation}
	and there exists $\alpha_0 >0$ such that for all $(t,t',s)\in B(\mathcal{M}_0,\delta_2)$,
	\begin{equation}\label{Eq:posdef je}
	\l e_{t,t'}, \La_{K\cup K'}(t,s)e_{t,t'} \r \geq \alpha_0.
	\end{equation}
	Since all elements in $\ep(K)$ and $\ep(K')$ are 1, we may write
	\begin{equation*}
	\begin{split}
	t &= (t_1, \ldots, t_m, t_{m+1}, \ldots, t_k, b_{k+1}, \ldots, b_{k+k'-m}, 0, \ldots, 0),\\
	t' &= (t'_1, \ldots, t'_m, b_{m+1}, \ldots, b_k, t'_{k+1}, \ldots, t'_{k+k'-m}, 0, \ldots, 0),
	\end{split}
	\end{equation*}
	where $t_i \in (a_i, b_i)$ for $i\in \sigma(K)$ and $t'_j \in (a_j, b_j)$ for $j \in \sigma(K')$. Therefore,
	\begin{equation}\label{Eq:e_k je}
	\begin{split}
	\l e_i , e_{t,t'}\r &\geq 0, \quad \forall \ m+1 \leq i \leq k,\\
	\l e_i , e_{t,t'}\r &\leq 0, \quad \forall \ k+1\leq i \leq k+k'-m,\\
	\l e_i , e_{t,t'}\r &= 0, \quad \forall \ k+k'-m< i \leq N.
	\end{split}
	\end{equation}
	Let
	\begin{equation}\label{Eq:D_k je}
	\begin{split}
	D_i &= \{ (t,t',s)\in B(\mathcal{M}_0,\delta_2): \alpha_i (t,t',s)\geq \beta_i \}, \quad \text{if}\ m+1 \leq i \leq k,\\
	D_i &= \{ (t,t',s)\in B(\mathcal{M}_0,\delta_2): \alpha_i (t,t',s)\leq -\beta_i \}, \quad \text{if}\ k+1\leq i \leq k+k'-m,\\
	D_0 &= \bigg\{ (t,t',s)\in B(\mathcal{M}_0,\delta_2): \sum_{i=1}^m \alpha_i (t,t',s)\l e_i , e_{t,t'}\r\geq \beta_0\bigg \},
	\end{split}
	\end{equation}
	where $\beta_0, \beta_1,\ldots, \beta_{k+k'-m}$ are positive constants such that $\beta_0 + \sum_{i=m+1}^{k+k'-m} \beta_i < \alpha_0$.
	It follows from (\ref{Eq:e_k je}) and (\ref{Eq:D_k je}) that, if $(t,s)$ does not belong to any of $D_0, D_{m+1}, \ldots, D_{k+k'-m}$, then by (\ref{Eq:orthnormal basis decomp je}),
	$$\l\La_{K\cup K'}(t,s)e_{t,t'}, e_{t,t'}\r = \sum_{i=1}^N \alpha_i(t,t',s) \l e_i , e_{t,t'}\r\leq \beta_0 + \sum_{i=m+1}^{k+k'-m} \beta_i < \alpha_0,$$
	which contradicts (\ref{Eq:posdef je}). Thus $D_0\cup \cup_{i=m+1}^{k+k'-m} D_i$ is a covering of $B(\mathcal{M}_0,\delta_2)$. By (\ref{Eq:cross term je}),
	\begin{equation*}
	\begin{split}
	\E&\{M_u^E (X,K) M_u^E (X,K') M_u^E (Y,L)\} \leq \int_{D_0} A(t,t',s)\,dtdt'ds + \sum_{i=m+1}^{k+k'-m} \int_{D_i}A(t,t',s)\,dtdt'ds.
	\end{split}
	\end{equation*}
	By the Kac-Rice metatheorem and the fact $\mathbbm{1}_{\{X(t)\geq u, Y(s)\geq u\}}\le \mathbbm{1}_{\{[X(t)+ Y(s)]/2\geq u\}}$, we obtain
	\begin{equation}\label{Eq:D0}
	\begin{split}
	\int_{D_0}&A(t,t',s)\,dtdt'ds
	\leq \int_{D_0} dtdt'ds \int_u^\infty dx \, p_{\nabla X_{|K}(t), \nabla X_{|K'} (t'), \nabla Y_{|L}(s)}(0,0,0)\\
	&\times p_{[X(t)+Y(s)]/2}(x|\nabla X_{|K}(t)=0, \nabla X_{|K'} (t')=0, \nabla Y_{|L}(s)=0)\\
	&\times \E \big\{ |\text{det} \nabla^2 X_{|K}(t) ||\text{det} \nabla^2 X_{|K'}(t') ||\text{det} \nabla^2 Y_{|L}(s) | \big| [X(t)+Y(s)]/2=x, \\
	&\qquad \quad \nabla X_{|K}(t)=\nabla X_{|K'}(t')=\nabla Y_{|L}(s)=0 \big\},
	\end{split}
	\end{equation}
	and that for $i= m+1, \ldots, k$,
	\begin{equation}\label{Eq:Di}
	\begin{split}
	&\quad \int_{D_i}A(t,t',s)\,dtdt'ds\\
	&\le \int_{D_i} dtdt'ds \int_u^\infty dx \int_0^\infty dw_i \, \E\big \{ |\text{det} \nabla^2 X_{|K}(t)| |\text{det} \nabla^2 X_{|K'}(t')| |\text{det} \nabla^2 Y_{|L}(s)|  \big|\\
	&\quad [X(t)+Y(s)]/2=x, \nabla X_{|K}(t)=0, X_i(t')=w_i, \nabla X_{|K'} (t')=\nabla Y_{|L}(s)=0 \big\}\\
	&\quad \times p_{[X(t)+Y(s)]/2, \nabla X_{|K}(t), X_i(t'), \nabla X_{|K'} (t'),\nabla Y_{|L}(s)}(x,0,w_i,0,0).
	\end{split}
	\end{equation}
	Comparing \eqref{Eq:D0} and \eqref{Eq:Di} with Eqs. (4.33) and (4.36) respectively in the proof of Theorem 4.8 in 
	\citet{ChengXiao2014}, the only essential difference is the additional effect of $\nabla Y_{|L}(s)=0$, which however 
	will not affect the desired super-exponentially small estimation since $(X, Y)$ is nondegenerate under the condition 
	({\bf H}2). Therefore, following similar arguments therein, we obtain that $\int_{D_0}A(t,t',s)\,dtdt'ds$ and 
	$\int_{D_i} A(t,t',s)\,dtdt'ds$ ($i= m+1, \ldots, k$) are super-exponentially small.
	
	It is similar to show that $\int_{D_i} A(t,t',s)\,dtdt'ds$ are super-exponentially small for $i=k+1, \ldots, k+k'-m$. For the 
	case $k=0$ or $l=0$, the argument is even simpler when applying the Kac-Rice formula (see for example 
	\eqref{Eq:factorial 1}). Hence the details are omitted here. We have completed the proof. 
\end{proof}

Notice that, in the proof of Lemma \ref{Lem:cross terms je}, we have shown in \eqref{Eq:Mu3} that, 
if $\mathcal{M}_0 = \emptyset$, then $\E\{M_u (X,K) M_u (X,K') M_u (Y,L)\}$ is super-exponentially small. 
Under the boundary condition \eqref{Eq:boundary} below, which generalizes the condition $\mathcal{M}_0 
= \emptyset$ in terms of the correlation function $r(t,s)$, we have the following result.
\begin{lemma}\label{Lem:cross terms je2} Let $\{(X(t),Y(s)): t\in T, s\in S\}$ be an $\R^2$-valued, centered, 
unit-variance Gaussian vector field satisfying $({\bf H}1)$ and $({\bf H}2)$. If for any faces $K_1\subset T$ 
and $L_1\subset S$,
	\begin{equation}\label{Eq:boundary}
	\begin{split}
	\Big\{(t,s)\in K_1 \times L_1 :\, r(t,s)=R, \prod_{i\notin \sigma(K_1)}\frac{\partial r}{\partial t_i}(t,s)\prod_{j\notin \sigma(L_1)}\frac{\partial r}{\partial s_j}(t,s)=0\Big\} = \emptyset,
	\end{split}
	\end{equation} 
	then there exists a constant $\alpha>0$ such that as $u\to \infty$,
	\begin{equation*}
	\begin{split}
	\E\{M_u (X,K) M_u (X,K') M_u (Y,L)\} &= o\Big(\exp \Big\{ -\frac{u^2}{1+R} -\alpha u^2 \Big\}\Big),\\
	\E\{M_u (X,K) M_u (Y,L) M_u (Y,L')\} &= o\Big(\exp \Big\{ -\frac{u^2}{1+R} -\alpha u^2 \Big\}\Big),
	\end{split}
	\end{equation*}
	where $K$ and $K'$ are adjacent faces of $T$, and $L$ and $L'$ are adjacent faces of $S$.
\end{lemma}
\begin{remark}
	In other words, the boundary condition \eqref{Eq:boundary} indicates that, for any point $(t,s)\in K_1 \times L_1$ attaining the maximum of correlation $R$, there must be $\frac{\partial r}{\partial t_i}(t,s)\neq 0$ for all $i\notin \sigma(K_1)$ and $\frac{\partial r}{\partial s_j}(t,s)\neq 0$ for all $j\notin \sigma(L_1)$. In particular, as an important property, we have that the boundary condition \eqref{Eq:boundary} implies the condition $({\bf H}3)$, as well as $\mathcal{M}_0=\emptyset$, where $\mathcal{M}_0$ defined in \eqref{Eq:M0}.
\end{remark}


\section{Estimation of the difference between EEC and the upper bound}\label{sec:diff}
In this section, we shall show that the difference between the expected number of extended outward local maxima, i.e. the upper bound in \eqref{Ineq:upperbound je}, and the expected Euler characteristic of the excursion set is super-exponentially small.
\begin{proposition}\label{Prop:simlify the high moment je} Let $\{(X(t),Y(s)): t\in T, s\in S\}$ be an $\R^2$-valued, centered, unit-variance Gaussian vector field satisfying $({\bf H}1)$, $({\bf H}2)$ and $({\bf H}3)$. Then there exists $\alpha>0$ such that for any $K\in \partial_k T$ and $L\in \partial_l S$ with $k, l\ge 0$, as $u\to \infty$,
	\begin{equation}\label{Eq:simplify-major-term}
	\begin{split}
	&\quad \E \{M_u^E (X,K)M_u^E (Y,L)\} \\
	&=(-1)^{k+l}\int_K\int_L \E\big\{{\rm det}\nabla^2 X_{|K}(t){\rm det}\nabla^2 Y_{|L}(s)\mathbbm{1}_{\{X(t)\geq u, \ \varepsilon^*_\ell X_\ell(t) \geq 0 \ {\rm for \ all}\ \ell\notin \sigma(K)\}}\\
	&\quad \times \mathbbm{1}_{\{Y(s)\geq u, \ \varepsilon^*_\ell Y_\ell(s) \geq 0 \ {\rm for \ all}\ \ell\notin \sigma(L)\}} \big| \nabla X_{|K}(t)=\nabla Y_{|L}(s)=0\big\}\\
	&\quad \times p_{\nabla X_{|K}(t), \nabla Y_{|L}(s)}(0,0) dtds +o\left( \exp \left\{ -\frac{u^2}{1+R} -\alpha u^2 \right\}\right)\\
	&= (-1)^{k+l} \E\bigg\{\bigg(\sum^k_{i=0} (-1)^i \mu_i(X,K)\bigg)\bigg(\sum^l_{j=0} (-1)^j \mu_j(Y,L)\bigg)\bigg\} + o\left( \exp \left\{ -\frac{u^2}{1+R} -\alpha u^2 \right\}\right).
	\end{split}
	\end{equation}
\end{proposition}

\begin{proof} The second equality in \eqref{Eq:simplify-major-term} follows from the application of the Kac-Rice theorem below:
	\begin{equation*}
		\begin{split}
			&\quad \E\bigg\{\bigg(\sum^k_{i=0} (-1)^i \mu_i(X,K)\bigg)\bigg(\sum^l_{j=0} (-1)^j \mu_j(Y,L)\bigg)\bigg\} \\
			&=\sum^k_{i=0} (-1)^i\sum^l_{j=0} (-1)^j\int_K\int_L dtds\, p_{\nabla X_{|K}(t), \nabla Y_{|L}(s)}(0,0)  \\
			&\quad \times \E\big\{|{\rm det}\nabla^2 X_{|K}(t)||{\rm det}\nabla^2 Y_{|L}(s)|\mathbbm{1}_{\{\text{index} (\nabla^2 X_{|K}(t))=i\}} \mathbbm{1}_{\{\text{index} (\nabla^2 X_{|L}(t))=j\}}  \\
			&\quad \times \mathbbm{1}_{\{X(t)\geq u, \ \varepsilon^*_\ell X_\ell(t) \geq 0 \ {\rm for \ all}\ \ell\notin \sigma(K)\}}\mathbbm{1}_{\{Y(s)\geq u, \ \varepsilon^*_\ell Y_\ell(s) \geq 0 \ {\rm for \ all}\ \ell\notin \sigma(L)\}} \big| \nabla X_{|K}(t)=\nabla Y_{|L}(s)=0\big\}\\
			 &=\int_K\int_L dtds\, p_{\nabla X_{|K}(t), \nabla Y_{|L}(s)}(0,0)  \E\big\{{\rm det}\nabla^2 X_{|K}(t){\rm det}\nabla^2 Y_{|L}(s)\\
			 &\quad \times \mathbbm{1}_{\{X(t)\geq u, \ \varepsilon^*_\ell X_\ell(t) \geq 0 \ {\rm for \ all}\ \ell\notin \sigma(K)\}}\mathbbm{1}_{\{Y(s)\geq u, \ \varepsilon^*_\ell Y_\ell(s) \geq 0 \ {\rm for \ all}\ \ell\notin \sigma(L)\}} \big| \nabla X_{|K}(t)=\nabla Y_{|L}(s)=0\big\},
		\end{split}
	\end{equation*}
where the last step is due to the fact $|{\rm det}\nabla^2 X_{|K}(t)|\mathbbm{1}_{\{\text{index} (\nabla^2 X_{|K}(t))=i\}}=(-1)^i{\rm det}\nabla^2 X_{|K}(t)$.
	
	To prove the first approximation in \eqref{Eq:simplify-major-term} and address the main idea, we first deal with a special case 
	when the two faces are both the interiors and then prove the general cases.
	
	\textbf{Case (i):  $\bm{k=l=N}$.} By the Kac-Rice metatheorem,
	\begin{equation*}
	\begin{split}
	&\quad \E \{M_u^E (X,K)M_u^E (Y,L)\}\\
	&=\int_K\int_L p_{\nabla X(t), \nabla Y(s)}(0,0) dtds \int_u^\infty \int_u^\infty dxdy\, p_{X(t),Y(s)}(x,y|\nabla X(t)=\nabla Y(s)=0)\\
	&\times \E\big\{{\rm det}\nabla^2 X(t){\rm det}\nabla^2 Y(s)\mathbbm{1}_{\{\nabla^2 X(t)\prec 0\}} \mathbbm{1}_{\{\nabla^2 Y(s)\prec 0\}} \big| X(t)=x, Y(s)=y, \nabla X(t)=\nabla Y(s)=0\big\}\\
	&:=\int_K\int_L p_{\nabla X(t), \nabla Y(s)}(0,0) dtds \int_u^\infty \int_u^\infty A(t,s,x,y)dxdy.
	\end{split}
	\end{equation*}
	Let
	\begin{equation}\label{Eq:O-Udelta}
	\begin{split}
	\mathcal{M}_1&=\{(t,s)\in \bar{K}\times \bar{L}: r(t,s)=R, \ \E\{X(t)\nabla Y(s)\}= \E\{Y(s)\nabla X(t)\}=0 \},\\
	B(\mathcal{M}_1, \delta_1)&=\left\{(t,s)\in K\times L: d\left((t,s), \mathcal{M}_1\right)\le\delta_1\right \},
	\end{split}
	\end{equation}
	where $\delta_1$ is a small positive number to be specified. Then, we only need to estimate
	\begin{equation}\label{Eq:A1}
		\begin{split}
	\int_{B(\mathcal{M}_1, \delta_1)} p_{\nabla X(t), \nabla Y(s)}(0,0) dtds \int_u^\infty \int_u^\infty A(t,s,x,y)dxdy,
\end{split}
\end{equation}
	since the integral above with $B(\mathcal{M}_1, \delta_1)$ replaced by $(K\times L)\backslash B(\mathcal{M}_1, \delta_1)$ 
	is super-exponentially small due to the fact 
	\[
	\sup_{(t,s)\in (K\times L)\backslash B(\mathcal{M}_1, \delta_1)} {\rm Var}([X(t)+ Y(s)]/2 | \nabla X(t)=\nabla Y(s)=0) < \frac{1+R}{2}.
	\]
	Notice that, for all $(t,s)\in \mathcal{M}_1$, $\E\{X(t)\nabla^2 X(t)\} \prec 0$ and $\E\{Y(s)\nabla^2 Y(s)\}\prec 0$ since $X(t)$ 
	and $Y(s)$ have unit-variance; and by $({\bf H}3)$ and Proposition \ref{prop:H3}, $\E\{X(t)\nabla^2 Y(s)\}\preceq 0$ and 
	$\E\{Y(s)\nabla^2 X(t)\}\preceq 0$. Thus there exists $\delta_1$ small enough such that $\E\{[X(t)+Y(s)] \nabla^2 Y(s)\}\prec 0$ 
	and $\E\{[X(t)+Y(s)] \nabla^2 X(t)\}\prec 0$ for all $(t,s)\in B(\mathcal{M}_1, \delta_1)$. In particular, let $\la_0$ be the largest 
        eigenvalue of $\E\{[X(t)+Y(s)] \nabla^2 X(t)\}$ over $B(\mathcal{M}_1, \delta_1)$, then $\la_0<0$ by the uniform continuity. 
	Also note that both $\E\{X(t)\nabla Y(s)\}$ and $\E\{Y(s)\nabla X(t)\}$ tend to 0 as $\delta_1\to 0$. Therefore, as $\delta_1\to 0$,
	\begin{equation}\label{Eq:CondHessianX}
	\begin{split}
	&\quad \E\{X_{ij}(t)|X(t)=x, Y(s)=y, \nabla X(t)=\nabla Y(s)=0\}\\
	&=(\E\{X_{ij}(t)X(t)\}, \E\{X_{ij}(t)Y(s)\}, \E\{X_{ij}(t)X_1(t)\}, \ldots, \E\{X_{ij}(t)X_N(t)\}, \\
	&\qquad  \E\{X_{ij}(t)Y_1(s)\}, \ldots, \E\{X_{ij}(t)Y_N(s)\})\left[ {\rm Cov}(X(t), Y(s), \nabla X(t), \nabla Y(s)) \right]^{-1} \\
	&\quad \cdot (x, y, 0, \ldots, 0, 0, \ldots, 0)^T\\
	&=(1+o(1))(\E\{X_{ij}(t)X(t)\}, \E\{X_{ij}(t)Y(s)\}) \left( \begin{array}{cc}
	1 & R \\
	R & 1 \end{array} \right)^{-1} \left( \begin{array}{c}
	x \\
	y \end{array} \right)\\
	&=(1+o(1))\frac{\E\{X_{ij}(t)X(t)\}[x-R y] + \E\{X_{ij}(t)Y(s)\}[y-R x]}{1-R^2};
	\end{split}
	\end{equation}
	and similarly,
	\begin{equation}\label{Eq:CondHessianY}
	\begin{split}
	&\quad \E\{Y_{ij}(s)|X(t)=x, Y(s)=y, \nabla X(t)=\nabla Y(s)=0\}\\
	&=(1+o(1))\frac{\E\{Y_{ij}(s)X(t)\}[x-R y] + \E\{Y_{ij}(s)Y(s)\}[y-R x]}{1-R^2}.
	\end{split}
	\end{equation}
	 By \eqref{Eq:CondHessianX} and \eqref{Eq:CondHessianY}, there exists $0<\ep_0<1-R$ such that for $\delta_1$ 
	 small enough and all $(x,y)\in [u,\infty)^2$ with $(\ep_0+R) x<y<(\ep_0+R)^{-1}x$ (so that $x-R y\ge \ep_0 u$ and $y-R x\ge \ep_0 u$),
	\begin{equation*}
	\begin{split}
	\Sigma_1(t,s,x,y) &:=\E\{\nabla^2 X(t)|X(t)=x, Y(s)=y, \nabla X(t)=\nabla Y(s)=0\}\prec 0 \text{ and}\\
	\Sigma_2(t,s,x,y) &:=\E\{\nabla^2 Y(s)|X(t)=x, Y(s)=y, \nabla X(t)=\nabla Y(s)=0\} \prec 0.
	\end{split}
	\end{equation*}
	Let $\Delta_1(t,s,x,y)=\nabla^2 X(t) - \Sigma_1(t,s,x,y)$ and $\Delta_2(t,s,x,y)=\nabla^2 Y(s) - \Sigma_2(t,s,x,y)$. 
	Due to the following decomposition,
	\begin{equation*}
		\begin{split}
	\{u\le x,y<\infty\} &= \{x\ge u, \, y\ge (\ep_0+R)^{-1}x\} \cup \{y\ge u, \, x\le (\ep_0+R)^{-1}y\}\\
	& \quad \cup \{x\ge u, \, u\vee (\ep_0+R) x< y< (\ep_0+R)^{-1}x\},
\end{split}
\end{equation*}
	we can write
	\begin{equation}\label{Eq:A}
	\begin{split}
	\int_u^\infty \int_u^\infty A(t,s,x,y)dxdy &=\int_u^\infty dx\int_{(\ep_0+R)^{-1}x}^\infty A(t,s,x,y)dy+ \int_u^\infty dy\int_{(\ep_0+R)^{-1}y}^\infty A(t,s,x,y)dx\\
	& \quad + \int_u^\infty dx\int_{u\vee (\ep_0+R) x}^{(\ep_0+R)^{-1}y} A(t,s,x,y)dx,
	\end{split}
	\end{equation}
	where the first two integrals on right are super-exponentially small since $(\ep_0+R)^{-1}>1$ and 
	\[
	\mathbbm{1}_{\{X(t)\ge u, Y(s)\ge (\ep_0+R)^{-1}X(t)\}} \vee \mathbbm{1}_{\{Y(s)\ge u, X(t)\ge (\ep_0+R)^{-1}Y(s)\}} 
	\le \mathbbm{1}_{\{[X(t)+Y(s)]/2\ge [1+(\ep_0+R)^{-1}]u/2\}}.
	\]
For the last integral in \eqref{Eq:A}, we have
	\begin{equation}\label{Eq:int_A}
	\begin{split}
	&\quad \int_u^\infty dx\int_{u\vee (\ep_0+R) x}^{(\ep_0+R)^{-1}x} A(t,s,x,y)dy\\
	&=\int_u^\infty dx\int_{u\vee (\ep_0+R) x}^{(\ep_0+R)^{-1}x} dy\, p_{X(t),Y(s)}(x,y|\nabla X(t)=\nabla Y(s)=0)\\
	&\times \E\big\{{\rm det}(\Delta_1(t,s,x,y)+ \Sigma_1(t,s,x,y)){\rm det}(\Delta_2(t,s,x,y)+ \Sigma_2(t,s,x,y))  \mathbbm{1}_{\{\Delta_1(t,s,x,y)+ \Sigma_1(t,s,x,y)\prec 0\}}\\
	&\times \mathbbm{1}_{\{\Delta_2(t,s,x,y)+ \Sigma_2(t,s,x,y)\prec 0\}}\big| X(t)=x, Y(s)=y, \nabla X(t)=\nabla Y(s)=0\big\}\\
	&:=\int_u^\infty dx\int_{u\vee (\ep_0+R) x}^{(\ep_0+R)^{-1}x} dy\, p_{X(t),Y(s)}(x,y|\nabla X(t)=\nabla Y(s)=0)E(t,s,x,y).
	\end{split}
	\end{equation}
Note that the following are two centered Gaussian random matrices (free of $x$ and $y$):
\begin{equation*}
	\begin{split}
		\Omega^X(t,s)&=(\Omega_{ij}^X(t,s))_{1\le i,j\le N}=(\Delta_1(t,s,x,y) | X(t)=x, Y(s)=y, \nabla X(t)=\nabla Y(s)=0),\\
		\Omega^Y(t,s)&=(\Omega_{ij}^Y(t,s))_{1\le i,j\le N}=(\Delta_2(t,s,x,y) | X(t)=x, Y(s)=y, \nabla X(t)=\nabla Y(s)=0).
	\end{split}
\end{equation*}
Denote the density of the Gaussian vector $((\Omega_{ij}^X(t,s))_{1\le i\le j\le N}, (\Omega_{ij}^Y(t,s))_{1\le i\le j\le N})$ by $h_{t,s}(v,w)$, where $v=(v_{ij})_{1\le i\le j\le N}$, $w=(w_{ij})_{1\le i\le j\le N}\in \R^{N(N+1)/2}$. Then
\begin{equation}\label{Eq:E(t,s,xy)}
	\begin{split}
E(t,s,x,y) &= \E\big\{{\rm det}(\Omega^X(t,s)+ \Sigma_1(t,s,x,y)){\rm det}(\Omega^Y(t,s)+ \Sigma_2(t,s,x,y))  \\
&\quad \times \mathbbm{1}_{\{\Omega^X(t,s)+ \Sigma_1(t,s,x,y)\prec 0\}} \mathbbm{1}_{\{\Omega^Y(t,s)+ \Sigma_2(t,s,x,y)\prec 0\}}\big\}\\
&= \int_{v:\, (v_{ij})+\Sigma_1(t,s,x,y)\prec 0} \int_{w:\, (w_{ij})+\Sigma_2(t,s,x,y)\prec 0} {\rm det}((v_{ij})+ \Sigma_1(t,s,x,y))\\
&\quad \times {\rm det}((w_{ij})+ \Sigma_2(t,s,x,y))h_{t,s}(v,w)dvdw,
\end{split}
\end{equation}
where $(v_{ij})$ and $(w_{ij})$ are respectively the abbreviations of the matrices $v=(v_{ij})_{1\le i, j\le N}$ and $w=(w_{ij})_{1\le i, j\le N}$. 
Recall that $x\wedge y\ge u$ and $(\ep_0+R) x<y<(\ep_0+R)^{-1}x$ implies $x-R y\ge \ep_0 u$ and $y-R x\ge \ep_0 u$. By 
\eqref{Eq:CondHessianX}, there exists a constant $0<c<-\la_0\ep_0/(1-R^2)$ such that for $\delta_1$ small enough and all 
$(t,s)\in B(\mathcal{M}_1, \delta_1)$, $x\ge u$ and $u\vee (\ep_0+R) x<y<(\ep_0+R)^{-1}x$,
\[
(v_{ij})+ \Sigma_1(t,s,x,y) \prec 0, \quad \forall \|(v_{ij})\|:=\Big(\sum_{i,j=1}^N v_{ij}^2\Big)^{1/2}<cu.
\]
Thus $\{v:\, (v_{ij})+\Sigma_1(t,s,x,y)\not\prec 0\} \subset \{v:\, \|(v_{ij})\|\ge cu\}$. This implies that the last integral in 
\eqref{Eq:E(t,s,xy)} with the integration domain replaced by $\{(v, w):\, (v_{ij})+\Sigma_1(t,s,x,y)\not\prec 0,\, w\in 
\R^{N(N+1)/2}\}$ is $o(e^{-\alpha'u^2})$ uniformly for all $(t,s)\in B(\mathcal{M}_1, \delta_1)$, where $\alpha'$ is 
a positive constant. The same result holds when replacing the integration domain by $\{(v, w):\, v\in \R^{N(N+1)/2}, 
\, (w_{ij})+\Sigma_2(t,s,x,y)\not\prec 0\}$. Therefore, we have that, uniformly for all $(t,s)\in B(\mathcal{M}_1, \delta_1)$, 
$x\ge u$ and $u\vee (\ep_0+R) x<y<(\ep_0+R)^{-1}x$,
\begin{equation*}
	\begin{split}
		E(t,s,x,y)& = \int_{\R^{N(N+1)/2}} \int_{\R^{N(N+1)/2}} {\rm det}((v_{ij})+ \Sigma_1(t,s,x,y))\\
		&\quad \times {\rm det}((w_{ij})+ \Sigma_2(t,s,x,y))h_{t,s}(v,w)dvdw + o(e^{-\alpha'u^2}).
	\end{split}
\end{equation*}
Plugging this into \eqref{Eq:int_A} and \eqref{Eq:A}, we obtain that the indicator functions $\mathbbm{1}_{\{\nabla^2 X(t)\prec 0\}}$ 
and $\mathbbm{1}_{\{\nabla^2 Y(s)\prec 0\}}$ in \eqref{Eq:A1} can be removed, causing only a super-exponentially small error. 
Therefore, there exists $\alpha>0$ such that for $u$ large enough,
	\begin{equation*}
	\begin{split}
	&\quad \E \{M_u^E (X,K)M_u^E (Y,L)\} \\
	&=\int_K\int_L p_{\nabla X(t), \nabla Y(s)}(0,0) dtds \int_u^\infty \int_u^\infty p_{X(t),Y(s)}(x,y|\nabla X(t)=\nabla Y(s)=0)\\
	&\quad \times  \E\{{\rm det}\nabla^2 X(t){\rm det}\nabla^2 Y(s)|X(t)=x, Y(s)=y, \nabla X(t)=\nabla Y(s)=0\} dxdy\\
	&\quad +o\Big( \exp \Big\{ -\frac{u^2}{1+R} -\alpha u^2 \Big\}\Big).
	\end{split}
	\end{equation*}

\textbf{Case (ii):  $\bm{k,l\ge 0}$.} Note that, if $k=0$ or $l=0$, then by the Kac-Rice formula, the terms in \eqref{Eq:simplify-major-term} 
involving the Hessian will vanish, making the proof easier. Therefore, without loss of generality, let $k, l\ge 1$, $\sigma(K)= \{1, \cdots, k\}$, 
$\sigma(L)= \{1, \cdots, l\}$ and assume all the elements in $\ep(K)$ and $\ep(L)$ are 1. By the Kac-Rice metatheorem,
	\begin{equation*}
	\begin{split}
	&\quad \E \{M_u^E (X,K)M_u^E (Y,L)\}\\
	&=(-1)^{k+l}\int_K\int_L p_{\nabla X_{|K}(t), \nabla Y_{|L}(s)}(0,0) dtds \int_u^\infty \int_u^\infty p_{X(t),Y(s)}\big(x,y\big |\nabla X_{|K}(t)=\nabla Y_{|L}(s)=0\big)\\
	&\quad  \E\big\{{\rm det}\nabla^2 X_{|K}(t){\rm det}\nabla^2 Y_{|L}(s)  \mathbbm{1}_{\{\nabla^2 X_{|K}(t)\prec 0\}} \mathbbm{1}_{\{\nabla^2 Y_{|L}(s)\prec 0\}}\mathbbm{1}_{\{X_{k+1}(t)>0,\ldots, X_N(t)>0\}}\\
	&\quad \times \mathbbm{1}_{\{Y_{l+1}(s)>0,\ldots, Y_N(s)>0\}} \big| X(t)=x, Y(s)=y, \nabla X_{|K}(t)=\nabla Y_{|L}(s)=0\big\}dxdy\\
	&:=(-1)^{k+l}\int_K\int_L p_{\nabla X_{|K}(t), \nabla Y_{|L}(s)}(0,0) dtds \int_u^\infty \int_u^\infty A'(t,s,x,y)dxdy.
	\end{split}
	\end{equation*}
Let
\begin{equation}\label{Eq:M2}
	\begin{split}
		\mathcal{M}_2&=\{(t,s)\in \bar{K}\times \bar{L}: r(t,s)=R, \ \E\{X(t)\nabla Y_{|L}(s)\}= \E\{Y(s)\nabla X_{|K}(t)\}=0 \},\\
		B(\mathcal{M}_2, \delta_2)&=\left\{(t,s)\in K\times L: d\left((t,s), \mathcal{M}_2\right)\le\delta_2\right \},
	\end{split}
\end{equation}
where $\delta_2$ is a small positive number to be specified. Then, we only need to estimate
\begin{equation}\label{Eq:A'}
	\begin{split}
\int_{B(\mathcal{M}_2, \delta_2)} p_{\nabla X_{|K}(t), \nabla Y_{|L}(s)}(0,0) dtds \int_u^\infty \int_u^\infty A'(t,s,x,y)dxdy,
	\end{split}
\end{equation}
since the integral above with $B(\mathcal{M}_2, \delta_2)$ replaced by $(K\times L)\backslash B(\mathcal{M}_2, \delta_2)$ 
is super-exponentially small due to the fact 
\[
\sup_{(t,s)\in (K\times L)\backslash B(\mathcal{M}_2, \delta_2)} {\rm Var}([X(t)+ Y(s)]/2 | \nabla X(t)=\nabla Y(s)=0) < \frac{1+R}{2}.
\]
On the other hand, following similar arguments in the proof for Case (i), we verify that removing the indicator functions 
$\mathbbm{1}_{\{\nabla^2 X_{|K}(t)\prec 0\}}$ and $\mathbbm{1}_{\{\nabla^2 Y_{|L}(s)\prec 0\}}$ in \eqref{Eq:A'} will  
only cause a super-exponentially small error. Combining these results, we have shown that the first approximation 
in \eqref{Eq:simplify-major-term} holds, completing the proof.	
\end{proof}

From the proof of Proposition \ref{Prop:simlify the high moment je}, we see that the same arguments can be applied 
to $\E \{M_u (X,K)M_u (Y,L)\}$, yielding the following result.
\begin{proposition}\label{Prop:simlify the high moment je2} Let $\{(X(t),Y(s)): t\in T, s\in S\}$ be an $\R^2$-valued, 
centered, unit-variance Gaussian vector field satisfying $({\bf H}1)$, $({\bf H}2)$ and $({\bf H}3 )$. Then there exists 
a constant $\alpha>0$ such that for any $K\in \partial_k T$ and $L\in \partial_l S$, as $u\to \infty$,
	\begin{equation*}
	\begin{split}
	&\quad \E \{M_u (X,K)M_u (Y,L)\} \\
	&=(-1)^{k+l}\int_K\int_L \E\big\{{\rm det}\nabla^2 X_{|K}(t){\rm det}\nabla^2 Y_{|L}(s)\mathbbm{1}_{\{X(t)\geq u, \ Y(s)\geq u\}}\big|\nabla X_{|K}(t)=\nabla Y_{|L}(s)=0\big\}\\
	&\quad  \times p_{\nabla X_{|K}(t), \nabla Y_{|L}(s)}(0,0) dtds +o\left( \exp \left\{ -\frac{u^2}{1+R} -\alpha u^2 \right\}\right)\\
	&=(-1)^{k+l} \E\bigg\{\bigg(\sum^k_{i=0} (-1)^i \widetilde{\mu}_i(X,K)\bigg)\bigg(\sum^l_{j=0} (-1)^j \widetilde{\mu}_j(Y,L)\bigg)\bigg\} + o\left( \exp \left\{ -\frac{u^2}{1+R} -\alpha u^2 \right\}\right).
	\end{split}
	\end{equation*}
\end{proposition}


\section{Proofs of the main results}\label{sec:proof}
\begin{proof}[Proof of Theorem \ref{Thm:MEC approximation je}]
	By Lemmas \ref{Lem:factorial moments je}, \ref{Lem:cross terms disjoint sets} and \ref{Lem:cross terms je}, together 
	with the fact $M_u^E (X,K)\le M_u (X,K)$, we obtain that the factorial moments and the last two sums in \eqref{Ineq:lowerbound je} 
	are super-exponentially small. It then follows from \eqref{Ineq:upperbound je} and \eqref{Ineq:lowerbound je} that, there exists 
	a constant $\alpha>0$ such that as $u\to \infty$,
	\begin{equation*}
		\begin{split}
			\P&\left\{\sup_{t\in T} X(t) \geq u, \sup_{s\in S} Y(s) \geq u \right\}\\
			&= \sum_{k,l=0}^N\sum_{K\in \partial_k T, \, L\in \partial_l S} \E \{M_u^E (X,K) M_u^E (Y,L) \}+ o\left( \exp \left\{ -\frac{u^2}{1+R} -\alpha u^2 \right\}\right).
		\end{split}
	\end{equation*}	
	The desired result is thus an immediate consequence of Proposition \ref{Prop:simlify the high moment je} and \eqref{Eq:Morse product}.
\end{proof}

\begin{proof}[Proof of Theorem \ref{Thm:MEC approximation je2}]
	 By Remark \ref{remark:M_u}, both inequalities \eqref{Ineq:upperbound je} and \eqref{Ineq:lowerbound je} still hold with $M_u^E(\cdot)$ 
	 replaced by $M_u(\cdot)$. Therefore, the corresponding factorial moments and the last two sums in \eqref{Ineq:lowerbound je} 
	 with $M_u^E(\cdot)$ replaced by $M_u(\cdot)$ are super-exponentially small by Lemmas \ref{Lem:factorial moments je}, 
	 \ref{Lem:cross terms disjoint sets} and \ref{Lem:cross terms je2}. Consequently, there exists a constant $\alpha>0$ such that as $u\to \infty$,
	 \begin{equation*}
	 	\begin{split}
	 		\P&\left\{\sup_{t\in T} X(t) \geq u, \sup_{s\in S} Y(s) \geq u \right\}\\
	 		&= \sum_{k,l=0}^N\sum_{K\in \partial_k T, \, L\in \partial_l S} \E \{M_u (X,K) M_u (Y,L) \}+ o\left( \exp \left\{ -\frac{u^2}{1+R} -\alpha u^2 \right\}\right).
	 	\end{split}
	 \end{equation*}	
	 The desired result is thus an immediate consequence of Proposition \ref{Prop:simlify the high moment je2} and \eqref{Eq:Morse product}.
\end{proof}

\begin{proof}[Proof of Theorem \ref{Thm:MEC approximation}]
	Note that, in the proof of Theorem \ref{Thm:MEC approximation je}, we have seen that the points in $\mathcal{M}_2$ defined in 
	\eqref{Eq:M2} make major contribution to the joint excursion probability. That is, with up to a super-exponentially small error, 
	we can focus only on those product faces, say $J\times F$, whose closure $\bar{J}\times \bar{F}$ contains the unique point 
	$(t^*, s^*)$ with $r(t^*, s^*)=R$ and satisfying $\sigma(J)\subset \mathcal{I}^R_X(t^*,s^*)$ and $\sigma(F)\subset \mathcal{I}^R_Y(t^*,s^*)$ 
	(i.e., the partial derivatives of $r$ are 0 at $(t^*, s^*)$ restricted on $J$ and $F$). Specifically, let
	\begin{equation*}
		\begin{split}
			T^* &=\{J\in \partial_k T: t^*\in \bar{J}, \, \sigma(J)\subset \mathcal{I}^R_X(t^*,s^*), \, k=0, \ldots, N\},\\
			S^* &=\{F\in \partial_\ell S: s^*\in \bar{F}, \, \sigma(F)\subset \mathcal{I}^R_Y(t^*,s^*), \, \ell=0, \ldots, N\};
		\end{split}
	\end{equation*}
and for each $J\in T^*$ and $F\in S^*$, let
\begin{equation*}
	\begin{split}
		M_u^{E^*} (X,J) & := \# \{ t\in J: X(t)\geq u, \nabla X_{|J}(t)=0,  \nabla^2 X_{|J}(t)\prec 0, \\
		& \qquad \qquad \qquad \varepsilon^*_jX_j(t) \geq 0 \ {\rm for \ all}\ j\in \mathcal{I}^R_X(t^*,s^*)\setminus \sigma(J) \},\\
		M_u^{E^*} (Y,F) & := \# \{ s\in F: Y(s)\geq u, \nabla Y_{|F}(s)=0,  \nabla^2 Y_{|F}(s)\prec 0, \\
		& \qquad \qquad \qquad \varepsilon^*_jY_j(s) \geq 0 \ {\rm for \ all}\ j\in \mathcal{I}^R_Y(t^*,s^*)\setminus \sigma(F) \}.
	\end{split}
\end{equation*}
Note that, both inequalities \eqref{Ineq:upperbound je} and \eqref{Ineq:lowerbound je} hold with $M_u^E(\cdot)$ replaced 
by $M_u^{E^*}(\cdot)$ when the corresponding face therein belongs to $T^*$ or $S^*$, and replaced by $M_u(\cdot)$ otherwise. 
Following similar arguments in deriving Theorems \ref{Thm:MEC approximation je} and \ref{Thm:MEC approximation je2}, 
we obtain that, there exists $\alpha>0$ such that as $u\to \infty$,
\begin{equation*}
	\begin{split}
		\P&\left\{\sup_{t\in T} X(t) \geq u, \sup_{s\in S} Y(s) \geq u \right\}\\
		&= \sum_{J\in T^*, \, F\in S^*} \E \{M_u^{E^*} (X,J) M_u^{E^*} (Y,F) \}+ o\left( \exp \left\{ -\frac{u^2}{1+R} -\alpha u^2 \right\}\right).
	\end{split}
\end{equation*}	
	The desired result then follows from Proposition \ref{Prop:simlify the high moment je}.
\end{proof}


\section{Examples}\label{sec:example} 
Throughout this section, we assume that $\{(X(t),Y(s)): t\in T, s\in S\}$, where $T=S=[0,1]$, is an $\R^2$-valued, 
centered, unit-variance Gaussian vector process satisfying $({\bf H}1)$, $({\bf H}2)$ and $({\bf H}3 )$.
\subsection{Example with correlation attaining the maximum at a unique point}
Suppose $r(t,s)$ attains the maximum $R$ only at a point $(t^*,s^*)$, i.e., $r(t^*,s^*)=R$. Let 
\begin{equation*}
\begin{split}
\la_1(t)&={\rm Var}(X'(t)), \ \la_2(s)={\rm Var}(Y'(s)),\
r_1(t,s)=\E\{X'(t)Y(s)\}, \ r_2(t,s)=\E\{X(t)Y'(s)\},\\
r_{11}(t,s)&=\E\{X''(t)Y(s)\}, \ r_{22}(t,s)=\E\{X(t)Y''(s)\},\  r_{12}(t,s)=\E\{X'(t)Y'(s)\}\\
\la_1 &= \la_1(t^*), \ \la_2 = \la_2(s^*), \
R_{11}=r_{11}(t^*,s^*), \ R_{22}=r_{22}(t^*,s^*),\  R_{12}=r_{12}(t^*,s^*).
\end{split}
\end{equation*}

\noindent \textbf{Case 1: $\bm {(t^*,s^*) = (0, 0)}$ and $\bm {r_1(0,0)r_2(0,0)\neq 0}$.} By Theorem \ref{Thm:MEC approximation je2},
\begin{equation*}
\begin{split}
\P\left\{\sup_{t\in T} X(t) \geq u, \sup_{s\in S} Y(s) \geq u \right\}
&= \P\{X(0)\geq u, Y(0) \geq u\} +o\left( \exp \left\{ -\frac{u^2}{1+R} -\alpha u^2 \right\}\right)\\
&= \frac{(1+R)^2}{2\pi \sqrt{1-R^2}}\frac{1}{u^2}e^{-\frac{u^2}{1+R}} (1+o(1)),
\end{split}
\end{equation*}
where the last line is due to a well-know asymptotics for $\P\{X(0)\geq u, Y(0) \geq u\}$, see \cite{LP00}.

\noindent \textbf{Case 2: $\bm {(t^*,s^*) = (0, 0)}$, $\bm{r_1(0,0)= 0}$ and $\bm{r_2(0,0)\neq 0}$.} By Theorem \ref{Thm:MEC approximation}, 
\begin{equation}\label{Eq:boundary-one-side}
\begin{split}
\P&\left\{\sup_{t\in T} X(t) \geq u, \sup_{s\in S} Y(s) \geq u \right\} \\
&= \P\{X(0)\geq u, Y(0) \geq u, X'(0)<0\} + I(u) + o\left( \exp \left\{ -\frac{u^2}{1+R} -\alpha u^2 \right\}\right),
\end{split}
\end{equation}
where
\begin{equation*}
\begin{split}
I(u)&= (-1)\int_0^1 p_{X'(t)}(0) dt \int_u^\infty \int_u^\infty  p_{X(t),Y(0)}(x,y|X'(t)=0)\\
&\quad \times  \E\{X''(t)|X(t)=x, Y(0)=y, X'(t)=0\} dxdy.
\end{split}
\end{equation*}
Since $X'(0)$ is independent of both $X(0)$ and $Y(0)$, we have
\begin{equation}\label{Eq:boundary-one-side-point}
\begin{split}
\P\{X(0)\geq u, Y(0) \geq u, X'(0)<0\}= \frac{(1+R)^2}{4\pi \sqrt{1-R^2}}\frac{1}{u^2}e^{-\frac{u^2}{1+R}} (1+o(1)).
\end{split}
\end{equation}
Let $\Sigma(t)=(\Sigma_{ij}(t))_{i,j=1,2}={\rm Cov}((X(t), Y(0))| X'(t)=0)$, implying $\Sigma_{11}(t)=1$, $\Sigma_{22}(t)=1-r_1^2(t,0)/\la_1(t)$ and $\Sigma_{12}(t)=\Sigma_{21}(t)=r(t,0)$.
Then
\begin{equation*}
\begin{split}
I(u)&=(-1)\int_0^1 \frac{1}{\sqrt{2\pi \la(t)}} dt \int_0^\infty \int_0^\infty \frac{e^{-\frac{1}{2}(x+u, \, y+u)\Sigma(t)^{-1}(x+u, \, y+u)^T}}{2\pi\sqrt{{\rm det}(\Sigma(t))}}\\
&\quad \times  \E\{X''(t)|X(t)=x+u, Y(0)=y+u, X'(t)=0\} dxdy,
\end{split}
\end{equation*}
where the expectation can be written as
$f(t)u + g(t)$ such that
\begin{equation*}
\begin{split}
f(0) = (-\la_1, R_{11}) \left( \begin{array}{cc}
1 & R \\
R & 1 \end{array} \right)^{-1} \left( \begin{array}{c}
1  \\
1 \end{array} \right) =\frac{R_{11}-\la_1}{1+R}.
\end{split}
\end{equation*}
By Theorem 7.5.3 in Tong (1990), as $u\to \infty$, the Mills ratio
\begin{equation}\label{Eq:Mills ratio}
\begin{split}
&\int_0^\infty \int_0^\infty e^{-\frac{1}{2}(x, y)\Sigma(t)^{-1}(x, y)^T - (u, u)\Sigma(t)^{-1}(x, y)^T} dxdy\\
&\sim \frac{1}{u^2[(\Sigma(t)^{-1})_{11}+(\Sigma(t)^{-1})_{21}][(\Sigma(t)^{-1})_{12}+(\Sigma(t)^{-1})_{22}]}.
\end{split}
\end{equation}
Therefore,
\begin{equation*}
\begin{split}
I(u)&\sim (-1)\int_0^1 \frac{1}{\sqrt{2\pi \la_1(t)}}\frac{1}{2\pi\sqrt{{\rm det}(\Sigma(t))}}f(t)\frac{1}{u}e^{-\frac{1}{2}u^2(1, 1)\Sigma(t)^{-1}(1, 1)^T} \\
&\quad \times \frac{1}{[(\Sigma(t)^{-1})_{11}+(\Sigma(t)^{-1})_{21}][(\Sigma(t)^{-1})_{12}+(\Sigma(t)^{-1})_{22}]}  dt.
\end{split}
\end{equation*}
It can be checked that the function
\begin{equation*}
\begin{split}
h(t):=\frac{1}{2}(1, 1)\Sigma(t)^{-1}(1, 1)^T=\frac{2-r_1(t,0)^2/\la_1(t) - 2r(t,0)}{2[1-r_1(t,0)^2/\la_1(t) - r(t,0)^2]}
\end{split}
\end{equation*}
attains its minimum only at $0$ with $h(0)=1/(1+R)$ and $h''(0)=R_{11}(R_{11}-\la_1)/[\la_1(1+R)^2]$. Applying the 
Laplace method (see, for example, Lemma A.3 in \citet{ChengXiao2014}), we obtain
\begin{equation}\label{Eq:boundary-one-side-edge}
\begin{split}
I(u)&\sim \frac{1}{2}\frac{1}{\sqrt{2\pi \la_1}}\frac{1}{2\pi\sqrt{1-R^2}}\frac{\la_1-R_{11}}{1+R}  (1+R)^2 \bigg(\frac{2\pi}{u^2}
\frac{\la_1(1+R)^2}{R_{11}(R_{11}-\la_1)}\bigg)^{1/2} \frac{1}{u}e^{-\frac{u^2}{1+R}}\\
&=\frac{\sqrt{\la_1-R_{11}}}{2\sqrt{-R_{11}}}\frac{(1+R)^2}{2\pi \sqrt{1-R^2}}\frac{1}{u^2}e^{-\frac{u^2}{1+R}}.
\end{split}
\end{equation}
Combining \eqref{Eq:boundary-one-side} with \eqref{Eq:boundary-one-side-point} and \eqref{Eq:boundary-one-side-edge}, 
we obtain
\begin{equation*}
\begin{split}
\P\left\{\sup_{t\in T} X(t) \geq u, \sup_{s\in S} Y(s) \geq u \right\}= \left(\frac{1}{2}+\frac{\sqrt{\la_1-R_{11}}}{2\sqrt{-R_{11}}}\right)
\frac{(1+R)^2}{2\pi \sqrt{1-R^2}}\frac{1}{u^2}e^{-\frac{u^2}{1+R}} (1+o(1)).
\end{split}
\end{equation*}

\noindent \textbf{Case 3: $\bm {(t^*,s^*) = (0, 0)}$ and $\bm{r_1(0,0)= r_2(0,0)\neq 0}$.} By Theorem \ref{Thm:MEC approximation}, 
\begin{equation}\label{Eq:boundary-two-side}
	\begin{split}
		&\P\Big\{\sup_{t\in T} X(t) \geq u, \sup_{s\in S} Y(s) \geq u \Big\}\\
		&= \P\{X(0)\geq u, Y(0) \geq u, X'(0)<0, Y'(0)<0\} \\
		&\quad + (-1)\int_0^1 p_{X'(t)}(0) dt \int_u^\infty \int_u^\infty \int_{-\infty}^0 p_{X(t),Y(0), Y'(0)}(x,y,z|X'(t)=0)\\
		&\qquad \times  \E\{X''(t)|X(t)=x, Y(0)=y, Y'(0)=z, X'(t)=0\} dxdydz\\
		&\quad + (-1)\int_0^1 p_{Y'(s)}(0) ds \int_u^\infty \int_u^\infty \int_{-\infty}^0 p_{X(0),Y(s), X'(0)}(x,y,z|Y'(s)=0)\\
		&\qquad \times  \E\{Y''(s)|X(0)=x, Y(s)=y, X'(0)=z, Y'(s)=0\} dxdydz\\
		&\quad +\int_0^1 \int_0^1 p_{X'(t),Y'(s)}(0,0) dtds \int_u^\infty \int_u^\infty p_{X(t),Y(s)}(x,y|X'(t)=Y'(s)=0)\\
		&\qquad \times  \E\{X''(t)Y''(s)|X(t)=x, Y(s)=y, X'(t)=Y'(s)=0\} dxdy\\
		&\quad + o\Big( \exp \Big\{ -\frac{u^2}{1+R} -\alpha u^2 \Big\}\Big)\\
		&:= I_1(u) + I_2(u) + I_3(u) + I_4(u) + o\Big( \exp \Big\{ -\frac{u^2}{1+R} -\alpha u^2 \Big\}\Big).
	\end{split}
\end{equation}
Since $(X'(0), Y'(0))$, which has covariance matrix ${\rm Var}(X'(0))=\la_1$, ${\rm Var}(Y'(0))=\la_2$ and $\E\{X'(0) Y'(0)\}=R_{12}$, 
is independent of $(X(0), Y(0))$, we have
\begin{equation}\label{Eq:boundary-one-side-point2}
	\begin{split}
		I_1(u) = \P\{X'(0)<0, Y'(0)<0\} \frac{(1+R)^2}{2\pi \sqrt{1-R^2}}\frac{1}{u^2}e^{-\frac{u^2}{1+R}} (1+o(1)).
	\end{split}
\end{equation}
Note that, if $R_{12}=0$, then $\P(X'(0)<0, Y'(0)<0)=1/4$. Similarly to \eqref{Eq:boundary-one-side-edge2}, we have
\begin{equation}\label{Eq:boundary-one-side-edge2}
	\begin{split}
		I_2(u)&\sim \frac{\sqrt{\la_1-R_{11}}}{2\sqrt{-R_{11}}}\frac{(1+R)^2}{2\pi \sqrt{1-R^2}}\frac{1}{u^2}e^{-\frac{u^2}{1+R}}(1+o(1)),\\
		  I_3(u)&\sim \frac{\sqrt{\la_2-R_{22}}}{2\sqrt{-R_{22}}}\frac{(1+R)^2}{2\pi \sqrt{1-R^2}}\frac{1}{u^2}e^{-\frac{u^2}{1+R}}(1+o(1)).
	\end{split}
\end{equation}

Let us compute $I_4$. Let $\Sigma(t,s)=(\Sigma_{ij}(t,s))_{i,j=1,2}={\rm Cov}((X(t), Y(s))| X'(t)=Y'(s)=0)$, implying
\begin{equation*}
	\begin{split}
		\Sigma_{11}(t,s)&= 1-\frac{\la_1(t) r_2^2(t,s)}{\la_1(t)\la_2(s)-r_{12}^2(t,s)},\quad 
		\Sigma_{22}(t,s)= 1-\frac{\la_2(s) r_1^2(t,s)}{\la_1(t)\la_2(s)-r_{12}^2(t,s)},\\
		\Sigma_{12}(t,s)&=\Sigma_{21}(t,s)= r(t,s)+\frac{r_{12}(t,s) r_1(t,s)r_2(t,s)}{\la_1(t)\la_2(s)-r_{12}^2(t,s)}.
	\end{split}
\end{equation*}
Then
\begin{equation}\label{Eq:I(u)-0-infty-1}
	\begin{split}
		I_4(u)&= \int_0^1\int_0^1 \frac{1}{2\pi\sqrt{\la_1(t)\la_2(s)-r_{12}^2(t,s)}}  dtds \int_0^\infty \int_0^\infty \frac{e^{-\frac{1}{2}(x+u, y+u)\Sigma(t,s)^{-1}(x+u, y+u)^T}}{2\pi\sqrt{{\rm det}(\Sigma(t,s))}}\\
		&\qquad \times  \E\{X''(t)Y''(s)|X(t)=x+u, Y(s)=y+u, X'(t)=Y'(s)=0\} dxdy.
	\end{split}
\end{equation}
where the expectation is on the product of two non-centered (conditional) Gaussian variables and hence its 
highest-order term in $u$ can be derived from the product of the means of Gaussian variables. We can write
$\E\{X''(t)|X(t)=x+u, Y(s)=y+u, X'(t)=Y'(s)=0\} = f(t,s)u + f_0(t,s,x,y)$ such that
\begin{equation*}
	\begin{split}
		f(0,0) = (-\la_1, R_{11}) \left( \begin{array}{cc}
			1 & R \\
			R & 1 \end{array} \right)^{-1} \left( \begin{array}{c}
			1  \\
			1 \end{array} \right) =\frac{R_{11}-\la_1}{1+R};
	\end{split}
\end{equation*}
and write $\E\{Y''(s)|X(t)=x+u, Y(s)=y+u, X'(t)=Y'(s)=0\} = g(t,s)u + g_0(t,s,x,y)$ such that
\begin{equation*}
	\begin{split}
		g(0,0) = (R_{22}, -\la_2) \left( \begin{array}{cc}
			1 & R \\
			R & 1 \end{array} \right)^{-1} \left( \begin{array}{c}
			1  \\
			1 \end{array} \right) =\frac{R_{22}-\la_2}{1+R}.
	\end{split}
\end{equation*}
Therefore, in the expectation in \eqref{Eq:I(u)-0-infty-1}, the highest-order term in $u$ evaluated at $(0,0)$ is given by 
$[(\la_1-R_{11})(\la_2-R_{22})/(1+R)^2]u^2$. Note that the Mills ratio in \eqref{Eq:Mills ratio} with $\Sigma(t)$ replaced 
by $\Sigma(t,s)$ is asymptotically $(1+R)^2/u^2$ at $(t,s)=(0,0)$. Plugging these into \eqref{Eq:I(u)-0-infty-1} yields
\begin{equation*}
	\begin{split}
		I_4(u)&\sim \int_0^1\int_0^1 \frac{1}{2\pi\sqrt{\la_1(t)\la_2(s)-r_{12}^2(t,s)}}\frac{1}{2\pi\sqrt{{\rm det}(\Sigma(t,s))}}f(t,s)g(t,s)u^2   \\
		&\qquad \times \frac{1}{u^2[(\Sigma(t,s)^{-1})_{11}+(\Sigma(t,s)^{-1})_{21}]^2} e^{-\frac{1}{2}u^2(1, 1)\Sigma(t,s)^{-1}(1, 1)^T} dtds.
	\end{split}
\end{equation*}
Since
\begin{equation*}
	\begin{split}
		h(t):=\frac{1}{2}(1, 1)\Sigma(t,s)^{-1}(1, 1)^T=\frac{1}{2}\frac{\Sigma_{11}(t,s)+\Sigma_{22}(t,s)-2\Sigma_{12}(t,s)}
		{\Sigma_{11}(t,s)\Sigma_{22}(t,s)- \Sigma_{12}^2(t,s)}
	\end{split}
\end{equation*}
attains its minimum only at $(t,s)=(0,0)$ with $h(0)=1/(1+R)$ and
\begin{equation*}
	\begin{split}
		&\nabla^2 h(0,0) 
		=\frac{1}{(1+R)^2(\la_1\la_2-R_{12}^2)} \left( \begin{array}{cc}
			(\la_1-R_{11})(R_{12}^2-\la_2R_{11}) & R_{12}(\la_1-R_{11})(R_{22}-\la_2) \\
			R_{12}(\la_1-R_{11})(R_{22}-\la_2) & (\la_2-R_{22})(R_{12}^2-\la_1R_{22}) \end{array} \right),\\
		&{\rm det}(\nabla^2 h(0,0)) = \frac{(\la_1-R_{11})(\la_2-R_{22})(R_{11}R_{22}-R_{12}^2)}{(1+R)^4(\la_1\la_2-R_{12}^2)}.
	\end{split}
\end{equation*}
Applying the Laplace method (see Lemma A.3 in \citep{ChengXiao2014}) yields
\begin{equation}\label{Eq:I4}
	\begin{split}
		I_4(u)&\sim \P(Z_1>0, Z_2>0)\frac{1}{2\pi\sqrt{\la_1\la_2-R_{12}^2}}\frac{1}{2\pi\sqrt{1-R^2}}\frac{(\la_1-R_{11})(\la_2-R_{22})}{(1+R)^2}u^2   \\
		&\qquad \times \frac{(1+R)^2}{u^2} \frac{2\pi}{u^2}\bigg(\frac{(1+R)^4(\la_1\la_2-R_{12}^2)}{(\la_1-R_{11})(\la_2-R_{22})(R_{11}R_{22}-R_{12}^2)}\bigg)^{1/2} e^{-\frac{u^2}{1+R}}\\
		&=\P(Z_1>0, Z_2>0)\frac{\sqrt{(\la_1-R_{11})(\la_2-R_{22})}}{\sqrt{R_{11}R_{22}-R_{12}^2}} \frac{(1+R)^2}{2\pi\sqrt{1-R^2}}\frac{1}{u^2}e^{-\frac{u^2}{1+R}},
	\end{split}
\end{equation}
where $(Z_1, Z_2)$ is a centered bivariate Gaussian variable with covariance $\nabla^2 h(0,0)$. Plugging \eqref{Eq:boundary-one-side-point2}, \eqref{Eq:boundary-one-side-edge2} and \eqref{Eq:I4} into \eqref{Eq:boundary-two-side}, we obtain
\begin{equation*}
	\begin{split}
		& \P\Big\{\sup_{t\in T} X(t) \geq u, \sup_{s\in S} Y(s) \geq u \Big\}= \Bigg[\P(X'(0)<0, Y'(0)<0)+\frac{\sqrt{\la_1-R_{11}}}{2\sqrt{-R_{11}}}  + \frac{\sqrt{\la_2-R_{22}}}{2\sqrt{-R_{22}}} \\
		&\quad+ \P(Z_1>0, Z_2>0)\frac{\sqrt{(\la_1-R_{11})(\la_2-R_{22})}}{\sqrt{R_{11}R_{22}-R_{12}^2}}\Bigg]\frac{(1+R)^2}{2\pi \sqrt{1-R^2}}\frac{1}{u^2}e^{-\frac{u^2}{1+R}}(1+o(1)),
	\end{split}
\end{equation*}
where the two probabilities on right become $1/4$ when $R_{12}=0$.

\noindent \textbf{Case 4: $\bm {(t^*,s^*) = (t^*, 0)}$, where $\bm{t^*\in (0,1)}$ and $\bm{r_2(t^*,0)\neq 0}$.} By Theorem 
\ref{Thm:MEC approximation} and similar arguments in Case 2, we obtain
\begin{equation*}
	\begin{split}
		\P\Big\{\sup_{t\in T} X(t) \geq u, \sup_{s\in S} Y(s) \geq u \Big\}= \frac{\sqrt{\la_1-R_{11}}}{\sqrt{-R_{11}}}\frac{(1+R)^2}
		{2\pi \sqrt{1-R^2}}\frac{1}{u^2}e^{-\frac{u^2}{1+R}}(1+o(1)).
	\end{split}
\end{equation*}

\noindent \textbf{Case 5: $\bm {(t^*,s^*) = (t^*, 0)}$, where $\bm{t^*\in (0,1)}$ and $\bm{r_2(t^*,0)= 0}$.} By Theorem 
\ref{Thm:MEC approximation} and similar arguments in Case 3, we obtain
\begin{equation*}
	\begin{split}
		\P\Big\{\sup_{t\in T} X(t) \geq u, \sup_{s\in S} Y(s) \geq u \Big\}&= \Bigg[\frac{\sqrt{\la_1-R_{11}}}{\sqrt{-R_{11}}}  
		+ \frac{\sqrt{(\la_1-R_{11})(\la_2-R_{22})}}{2\sqrt{R_{11}R_{22}-R_{12}^2}}\Bigg]\\
		&\quad \times \frac{(1+R)^2}{2\pi \sqrt{1-R^2}}\frac{1}{u^2}e^{-\frac{u^2}{1+R}}(1+o(1)).
	\end{split}
\end{equation*}

\noindent \textbf{Case 6:} $\bm {(t^*,s^*)\in (0,1)^2}$. By Theorem \ref{Thm:MEC approximation} and similar arguments in Case 3, we obtain
\begin{equation*}
	\begin{split}
		\P\Big\{\sup_{t\in T} X(t) \geq u, \sup_{s\in S} Y(s) \geq u \Big\}= \frac{\sqrt{(\la_1-R_{11})(\la_2-R_{22})}}{\sqrt{R_{11}R_{22}-R_{12}^2}} \frac{(1+R)^2}{2\pi \sqrt{1-R^2}}\frac{1}{u^2}e^{-\frac{u^2}{1+R}}(1+o(1)).
	\end{split}
\end{equation*}

\subsection{Examples with correlation attaining the maximum on a line}
Here we consider the bivariate Gaussian random fields in \citet{ZhouXiao07}, where the smooth case was not studied 
since the double sum method therein is not applicable. Let $X(t)$ and $Y(s)$ be smooth stationary Gaussian processes 
with covariances satisfying
\begin{equation*}
	\begin{split}
		\E\{X(0)X(t)\} &= 1- \frac{\la_1}{2}|t|^2 (1+o(1)), \quad {\rm as } \ |t| \to 0,\\
		\E\{Y(0)Y(s)\} &= 1- \frac{\la_2}{2}|s|^2 (1+o(1)), \quad {\rm as } \ |s| \to 0,
	\end{split}
\end{equation*}
which implies ${\rm Var}(X'(t))=-\E\{X(t)X''(t)\}=\la_1$ and ${\rm Var}(Y'(s))=-\E\{Y(s)Y''(s)\}=\la_2$. Assume that the correlation of $X$ and $Y$ satisfies
\[
r(t,s) = \E\{X(t)Y(s)\} =  \rho(|t-s|), \quad \forall t,s\in [0,1],
\]
where $\rho$ is a real function. Suppose $\rho$ attains its maximum $R$ only at 0 with $\rho'(0)=0$ and $\rho''(0)<0$. 
This indicates that the maximum correlation $R$ is only achieved on the diagonal line $\{t=s: 0\le t,s\le 1\}$. By 
Theorem \ref{Thm:MEC approximation je}, we have
\begin{equation*}
	\begin{split}
		&\quad \P\Big\{\sup_{t\in [0,1]} X(t) \geq u, \sup_{s\in [0,1]} Y(s) \geq u \Big\}\\
		&=\int_0^1 \int_0^1 p_{X'(t),Y'(s)}(0,0) dtds \int_u^\infty \int_u^\infty p_{X(t),Y(s)}(x,y|X'(t)=Y'(s)=0)\\
		&\qquad \times  \E\{X''(t)Y''(s)|X(t)=x, Y(s)=y, X'(t)=Y'(s)=0\} dxdy\\
		&\quad +\P\{X(0)\geq u, Y(0) \geq u, X'(0)<0, Y'(0)<0\} \\
		&\quad + \P\{X(1)\geq u, Y(1) \geq u, X'(1)>0, Y'(1)>0\} + o\Big( \exp \Big\{ -\frac{u^2}{1+R} -\alpha u^2 \Big\}\Big)\\
		&= I(u)(1+o(1)),
	\end{split}
\end{equation*}
where $I(u)$ denotes the integral term in the second and third lines. We shall derive below the asymptotics of
$I(u)$ which gives the highest-order term in $u$. By the stationarity and change of variables (using $z=s$ and $w=t-s$ 
for $0<s<t<1$ and the symmetry property),
\begin{equation*}
	\begin{split}
		I(u)&= \int_0^1 \int_0^1 p_{X'(0),Y'(|t-s|)}(0,0) dtds \int_u^\infty \int_u^\infty p_{X(0),Y(|t-s|)}(x,y|X'(0)=Y'(|t-s|)=0)\\
		&\qquad \times  \E\{X''(0)Y''(|t-s|)|X(0)=x, Y(|t-s|)=y, X'(0)=Y'(|t-s|)=0\} dxdy\\
		&= 2\int_0^1 (1-t)p_{X'(0),Y'(t)}(0,0) dt \int_u^\infty \int_u^\infty p_{X(0),Y(t)}(x,y|X'(0)=Y'(t)=0)\\
		&\qquad \times  \E\{X''(0)Y''(t)|X(0)=x, Y(t)=y, X'(0)=Y'(t)=0\} dxdy \\
		&:= 2I_0(u).
	\end{split}
\end{equation*}
Let $\Sigma(t)=(\Sigma_{ij}(t))_{i,j=1,2}={\rm Cov}((X(0), Y(t))| X'(0)=Y'(t)=0)$, implying
\begin{equation*}
	\begin{split}
		\Sigma_{11}(t)&= 1-\frac{\la_1 \rho'(t)^2}{\la_1\la_2-\rho''(t)^2},\quad 
		\Sigma_{22}(t)= 1-\frac{\la_2 \rho'(t)^2}{\la_1\la_2-\rho''(t)^2},\\
		\Sigma_{12}(t)&=\Sigma_{21}(t)= \rho(t)+\frac{\rho''(t)\rho'(t)^2}{\la_1\la_2-\rho''(t)^2}.
	\end{split}
\end{equation*}
Then
\begin{equation}\label{Eq:I(u)-0-infty-2}
	\begin{split}
		I_0(u)&= \int_0^1 \frac{1-t}{2\pi\sqrt{\la_1\la_2-\rho''(t)^2}}  dt \int_0^\infty \int_0^\infty \frac{1}{2\pi\sqrt{{\rm det}(\Sigma(t))}}
		e^{-\frac{1}{2}(x+u, y+u)\Sigma(t)^{-1}(x+u, y+u)^T}\\
		&\qquad \times  \E\{X''(0)Y''(t)|X(0)=x+u, Y(t)=y+u, X'(0)=Y'(t)=0\} dxdy.
	\end{split}
\end{equation}
We have $\E\{X''(0)|X(0)=x+u, Y(t)=y+u, X'(0)=Y'(t)=0\} = f(t)u + f_0(t,x,y)$ with
\begin{equation*}
	\begin{split}
		f(0) = (-\la_1, \rho''(0)) \left( \begin{array}{cc}
			1 & R \\
			R & 1 \end{array} \right)^{-1} \left( \begin{array}{c}
			1  \\
			1 \end{array} \right) =\frac{\rho''(0)-\la_1}{1+R};
	\end{split}
\end{equation*}
and $\E\{Y''(t)|X(0)=x+u, Y(t)=y+u, X'(0)=Y'(t)=0\}= g(t)u + g_0(t,x,y)$ with
\begin{equation*}
	\begin{split}
		g(0) = (\rho''(0), -\la_2) \left( \begin{array}{cc}
			1 & R \\
			R & 1 \end{array} \right)^{-1} \left( \begin{array}{c}
			1  \\
			1 \end{array} \right) =\frac{\rho''(0)-\la_2}{1+R}.
	\end{split}
\end{equation*}
Therefore, in the expectation in \eqref{Eq:I(u)-0-infty-2}, the highest-order term in $u$ evaluated at $t=0$ is given by $[(\la_1-\rho''(0))
(\la_2-\rho''(0))/(1+R)^2]u^2$. Note that the Mills ratio in \eqref{Eq:Mills ratio} is asymptotically $(1+R)^2/u^2$ at $t=0$. Plugging these 
into \eqref{Eq:I(u)-0-infty-2} yields
\begin{equation*}
	\begin{split}
		I_0(u)&\sim \int_0^1 \frac{1-t}{2\pi\sqrt{\la_1\la_2-\rho''(t)^2}}\frac{1}{2\pi\sqrt{{\rm det}(\Sigma(t))}}f(t)g(t)   \\
		&\qquad \times \frac{1}{[(\Sigma(t)^{-1})_{11}+(\Sigma(t)^{-1})_{12}]^2} e^{-\frac{1}{2}u^2(1, 1)\Sigma(t)^{-1}(1, 1)^T} dt.
	\end{split}
\end{equation*}
Since
\begin{equation*}
	\begin{split}
		h(t):=\frac{1}{2}(1, 1)\Sigma(t)^{-1}(1, 1)^T=\frac{1}{2}\frac{\Sigma_{11}(t)+\Sigma_{22}(t)-2\Sigma_{12}(t)}{\Sigma_{11}(t)\Sigma_{22}(t)- \Sigma_{12}(t)^2}
	\end{split}
\end{equation*}
attains its minimum only at $t=0$ with $h(0)=1/(1+R)$ and
\begin{equation*}
	\begin{split}
		h''(0)=\frac{-\rho''(0)(\la_1-\rho''(0))(\la_2-\rho''(0))}{(1+R)^2[\la_1\la_2-\rho''(0)^2]}.
	\end{split}
\end{equation*}
Applying the Laplace method (see Lemma A.3 in \citep{ChengXiao2014}) yields
\begin{equation*}
	\begin{split}
		I_0(u)&\sim \frac{1}{2}\frac{1}{2\pi\sqrt{\la_1\la_2-\rho''(0)^2}}\frac{1}{2\pi\sqrt{1-R^2}}\frac{[\la_1-\rho''(0)][\la_2-\rho''(0)]}{(1+R)^2}u^2   \\
		&\qquad \times \frac{(1+R)^2}{u^2} \bigg(-\frac{2\pi}{u^2}\frac{(1+R)^2[\la_1\la_2-\rho''(0)^2]}{\rho''(0)(\la_1-\rho''(0))(\la_2-\rho''(0))}\bigg)^{1/2} e^{-\frac{u^2}{1+R}}\\
		&=\frac{1}{2}\frac{\sqrt{(\la_1-\rho''(0))(\la_2-\rho''(0))}(1+R)}{(2\pi)^{3/2}\sqrt{1-R^2}\sqrt{-\rho''(0)}} \frac{1}{u}e^{-\frac{u^2}{1+R}}.
	\end{split}
\end{equation*}
Thus we obtain
\begin{equation*}
	\begin{split}
		\P&\Big\{\sup_{t\in [0,1]} X(t) \geq u, \sup_{s\in [0,1]} Y(s) \geq u \Big\}\\
		&= \frac{1}{(2\pi)^{3/2}} \sqrt{\frac{(\la_1-\rho''(0))(\la_2-\rho''(0))(1+R)}{-\rho''(0)(1-R)}}\frac{1}{u} e^{-\frac{u^2}{1+R}}(1+o(1)).
	\end{split}
\end{equation*}

\bibliographystyle{plainnat}

\begin{thebibliography}{36}
	
	\bibitem[Adler(2000)]{Adler00}
	Adler, R. J. (2000).
	\newblock On excursion sets, tube formulas and maxima of random fields.
	\newblock {\it Ann. Appl. Probab.} {\bf 10}, 1--74.
	
	\bibitem[Adler and Taylor(2007)]{Adler:2007}
	Adler, R. J. and Taylor, J. E. (2007).
	\newblock \emph{Random fields and geometry}.
	\newblock Springer, New York.
	
	\bibitem[Anshin(2006)]{Anshin06}
	Anshin, A. B. (2006).
	\newblock On the probability of simultaneous extremes of two Gaussian nonstationary processes.
	\newblock {\it Theory Probab. Appl.} {\bf 50}, 353--366.
	
	\bibitem[Aza\"is and Delmas(2002)]{AzaisD02}
	Aza\"is, J. M. and Delmas, C. (2002).
	\newblock Asymptotic expansions for the distribution of the
	maximum of Gaussian random fields.
	\newblock {\it Extremes.} {\bf 5}, 181--212.
	
	\bibitem[Aza\"is and Wschebor(2008)]{AzaisW08}
	Aza\"is, J. M. and Wschebor, M. (2008).
	\newblock A general expression for the distribution of the maximum of a
	Gaussian field and the approximation of the tail.
	\newblock {\it Stoch. Process. Appl.} {\bf 118},  1190--1218.
	
	\bibitem[Aza\"is and Wschebor(2009)]{AzaisW09}
	Aza\"is, J. M. and Wschebor, M. (2009).
	\newblock { \it Level Sets and Extrema of Random Processes
		and Fields}.
	\newblock John Wiley \& Sons, Hoboken, NJ.
	
	\bibitem[Cheng and Xiao (2016)]{ChengXiao2014}
	Cheng, D. and Xiao, Y. (2016).
	\newblock The mean Euler characteristic and excursion probability of Gaussian random fields with stationary increments.
	\newblock {\it Ann. Appl. Probab.} {\bf 26},  722--759.
	
	\bibitem[Debicki et al.(2010)]{DKMR10}
	Debicki, K., Kosi\'nski, K. M., Mandjes, M. and Rolski, T. (2010).
	\newblock Extremes of multidimensional Gaussian processes.
	\newblock {\it Stoch. Process. Appl.} {\bf 120}, 2289--2301.
	
	\bibitem[Debicki et al.(2015)]{Debicki:2015}
	Debicki, K., Hashorva, E., Ji, L. and Tabi\'s, K. (2015).
	\newblock Extremes of vector-valued Gaussian processes: Exact asymptotics.
	\newblock {\it Stoch. Process. Appl.} {\bf 125}, 4039--4065.
	
	\bibitem[Hashorva and Ji(2014)]{HashorvaJi2014}
 Hashorva, E. and Ji, L. (2014).
\newblock Extremes and first passage times of correlated fractional {B}rownian motions. 
  \newblock{\it Stoch. Models} {\bf 30}, {272--299}.


	\bibitem[Ladneva and Piterbarg(2000)]{LP00}
	Ladneva, A. and Piterbarg, V. I. (2000).
	\newblock On double extremes of Gaussian stationary processes.
	\newblock {\it website: http://www.eurandom.tue.nl/reports/2000/027-report.pdf}
	
	\bibitem[Piterbarg(1996)]{Piterbarg:1996}
	Piterbarg, V. I. (1996).
	\newblock \emph{Asymptotic Methods in the Theory of Gaussian Processes and
		Fields. Translations of Mathematical Monographs 148}.
	\newblock Amer. Math. Soc., Providence, RI.
	
	\bibitem[Piterbarg (1996)]{Piterbarg96}
	Piterbarg, V. I. (1996)
	\newblock Rice's method for large excursions of Gaussian random fields.
	\newblock Technical Report NO. 478, Center for Stochastic Processes, Univ. North Carolina.
	
	\bibitem[Piterbarg and Stamatovic(2005)]{PS05}
	Piterbarg, V. I. and Stamatovic, S. (2005).
	\newblock Crude asymptotics of the probability of simultaneous
	high extrema of two Gaussian processes: the dual action function.
	\newblock {\it Russ. Math. Surv.} {\bf 60}, 167--168.
	
	\bibitem[Sun (2001)]{Sun:2001}
	Sun, J. (2001).
	\newblock Multiple comparisons for a large number of parameters.
	\newblock {\it Biometrical Journal} {\bf 43},  627--643.
	
	\bibitem[Taylor and Adler(2003)]{TaylorAdler03}
	Taylor, J. E. and Adler, R. J. (2003).
	\newblock  Euler characteristics for Gaussian fields on manifolds.
	\newblock {\it  Ann. Probab.} {\bf 31},  533--563.
	
	\bibitem[Taylor et al.(2005)]{TTA05}
	Taylor, J. E., Takemura, A. and Adler, R. J. (2005).
	\newblock Validity of the expected Euler characteristic heuristic.
	\newblock {\it Ann. Probab.} {\bf 33}, 1362--1396.
	
	\bibitem[Tong(1990)]{Tong1990}
	Tong, Y. L. (1990).
	\newblock {\it The Multivariate Normal Distribution}.
	\newblock Springer, New York.
	
	\bibitem[Wong(2001)]{Wong2001}
	Wong, R. (2001).
	\newblock {\it Asymptotic Approximations of Integrals}.
	\newblock SIAM, Philadelphia, PA.
	
	\bibitem[Zhou and Xiao(2017)]{ZhouXiao07}
	Zhou, Y. and Xiao, Y. (2017).
	\newblock  Tail asymptotics for the extremes of bivariate Gaussian random fields.
	\newblock {\it  Bernoulli.} {\bf 23},  1566--1598.
	
\end{thebibliography}
\begin{small}

\end{small}

\begin{quote}
	\begin{small}
		
		\textsc{Dan Cheng}\\
		School of Mathematical and Statistical Sciences\\
		Arizona State University\\
		900 S Palm Walk\\
		Tempe, AZ 85281, USA\\
		E-mail: \texttt{cheng.stats@gmail.com}

		\vspace{.1in}
		
		\textsc{Yimin Xiao}\\
		Department of Statistics and Probability \\
		Michigan State University\\
		619 Red Cedar Road, C-413 Wells Hall\\
		East Lansing, MI 48824, USA\\
		E-mail: \texttt{xiaoy@msu.edu}

	\end{small}
\end{quote}

\end{document}